\documentclass[12pt]{article}
\usepackage[utf8]{inputenc}

\usepackage{amssymb,   enumitem,   xcolor}
\usepackage{tikz}
\usepackage[normalem]{ulem}

\usepackage{graphics,   graphicx,   amsmath}
\usepackage{lmodern,   microtype}

\usepackage{mathtools,   amsthm,   icomma,   upgreek,   xfrac,   amsmath,   dsfont,   float}
\mathtoolsset{showonlyrefs,   mathic}

\usepackage[colorlinks=true,    urlcolor=blue,    linkcolor=blue,    citecolor=blue]{hyperref}

\theoremstyle{plain}
\newtheorem{theorem}{Theorem}[section]
\newtheorem{lemma}[theorem]{Lemma}
\newtheorem{proposition}[theorem]{Proposition}

\theoremstyle{definition}

\newtheorem{corollary}{Corollary}[section]

\newtheorem{remark}[theorem]{Remark}
\newtheorem{fact}{Fact}[section]

\numberwithin{equation}{section}

\newcommand{\E}{\mathbf{E}}
\renewcommand{\P}{{P}}
\newcommand{\qq}{{\bf P}}

\newcommand{\R}{\mathbf{R}}
\newcommand{\N}{\mathbf{N}}

\newcommand{\h}{\H}

\newcommand{\X}{X^{(r)}_n}
\newcommand{\Xe}{X^{(2)}_n}
\newcommand{\Ze}{Z^{(2)}_n}

\newcommand{\Var}{\text{Var}}

\def\ind{{\mathbb I}}

\newcommand{\cov}{\mathrm{Cov}}
\def\[{\left[}
\def\]{\right]}
\def\real{{\mathord{\mathbb R}}}

\def\Var{\mathrm{Var}}
\newcommand{\pr}{\mathbb{P}}

\newcommand{\V}{\mathcal{V}}

\newcommand{\A}{\mathcal{A}}

\newcommand{\Subgraphs}{{\cal K}}

\newcommand{\covers}{{\mathcal C}}
\newcommand{\cover}{{\mathbb C}}

\newcommand{\kli}{\mathbf{K}}

\newcommand{\Pra}[1]{\pr\left(#1\right)}
\newcommand{\Gnmp}{{\mathcal G}(n, m, p)}
\newcommand{\cGnmp}{{\overline{\mathcal G}}(n, m, p)}
\newcommand{\Gnp}{{\mathcal G}(n, p)}

\newcommand{\hatp}{\mathbf{p}}

\newcommand{\xx}{{\bf x}}
\newcommand{\yy}{{\bf y}}
\newcommand{\zz}{{\bf z}}

\newcommand{\ttt}{{\bf t}}
\newcommand{\II}{{\bf I}}
\newcommand{\JJ}{{\bf J}}
\newcommand{\indeksy}{{\boldsymbol{\Gamma}}}

\newcommand{\pp}{{\bf p}}
\newcommand{\calP}{{\cal P}}

\newcommand*{\db}[1]{\overline{\overline{#1}}}

\allowdisplaybreaks

\date{\vspace{-5ex}}

\title{
\huge
Normal approximation for subgraph counts\\ in random intersection graphs: the case of edges
} 

\author{ 
Katarzyna Rybarczyk\thanks{Faculty of Mathematics and Computer Science,      Adam Mickiewicz University,      Pozna\'n,      Poland.
e-mail: {\tt kryba@amu.edu.pl}. ORCID: 0000-0002-3998-6482}
\and
Grzegorz Serafin\thanks{Faculty of Pure and Applied Mathematics,       Wroc{\l}aw University of Science and Technology,      Ul. Wybrze\.ze Wyspia\'nskiego 27,       Wroc{\l}aw,      Poland.
e-mail: {\tt grzegorz.serafin@pwr.edu.pl}. ORCID: 0000-0003-3170-6850}}

\begin{document}
\maketitle

\begin{abstract} 
We investigate the subgraph count in the random intersection graph model $\mathcal G(n,m,p)$. Due to substantial edge dependencies, studying even such fundamental statistics is significantly more challenging than in the classical binomial model $\mathcal G(n,p)$. This article establish new normal approximation bounds in both the Wasserstein and Kolmogorov distances for general non-degenerate $U$-statistics. We then illustrate how these results can be potentially applied to the standardized number of subgraphs in the random intersection graph model. Finally, as a basic example, we derive necessary and sufficient conditions for the asymptotic normality of the number of edges in $\mathcal G(n,m,p)$. 
\end{abstract}
\noindent\emph{MSC2020 subject classification}: Primary 05C80,    60F05; Secondary 05C82.

\section{Introduction}
Counting small subgraphs  is one of the fundamental problems in the study of random graphs. In particular,   it plays a crucial role in the research on the classical $\Gnp$ model (see \cite{BKR,   FK} and references therein) and still attracts a significant interest \cite{CD,   ER,   GGS,   HMS,   PSbej}. The long standing problem of normal approximation bounds in Kolmogorov distance has been resolved just recently \cite{PSbej},   see also \cite{ER,   R,   Z} for different approaches. 
In this article,  we address the small subgraph count in  the random intersection graph model $\Gnmp$,  that   was introduced by Karo\'nski,   Scheinerman and Singer--Cohen in \cite{1999KarSchSin}.  Unlike  the classical binomial Erd\H{o}s--R\'enyi random graph,    $\Gnmp$ exhibits strong dependencies between edges,   which is manifested by  high density of cliques. This abundance of cliques,  along with the edge correlations and other structural properties   of random intersection graphs have resulted in the model and its generalisations finding many applications such as: ``gate matrix layout'' for VLSI design \cite{1999KarSchSin},   cluster analysis and classification \cite{RIGGodehardt1},   analysis of complex networks  \cite{RIGClustering2,   RIGTunableDegree},   secure wireless networks  \cite{WSNphase2} or epidemics \cite{GpEpidemics}. For more details,   we refer the reader to the survey papers \cite{RIGsurvey1,  RIGsurvey2, Spirakis2013}.

Due to the  edge dependencies,   studying random intersection graph in the context of small subgraphs turns out to be much more challenging comparing to the binomial model. Even the analysis of the distribution of the number of edges in $\Gnmp$ is non-trivial,   whereas in the $\Gnp$ model   the number of edges follows a simple binomial distribution. This complexity of the model has been already revealed in the form of  thresholds for existence of at least one copy   of a given graph \cite{1999KarSchSin}. Concerning further results in this direction,    the distribution of the number of copies of small subgraphs in $\Gnmp$ on the threshold was studied in \cite{2010RybSta} and \cite{2017RybSta}. For some related random intersection graph models in various ranges of parameters a normal approximation of the number of triangles \cite{Bloznelis2018} and 2--connected subgraphs \cite{Bloznelis2023} was determined. However,   none of the methods used in these works could be adopted to determine conditions for asymptotic normality of the number of small cliques in $\Gnmp$. 
Only recently some results concerning the normal distribution of the number of triangles in a random intersection graph have been proved by Dong and Hu \cite{Dong2023}. They employed the Stein method to find an upper bound on the Kolmogorov   distance between standardised random variable counting the number of triangles in $\Gnmp$ and the standard normal distribution. Although \cite{Dong2023}  is a substantial contribution to the solution of the problem,   it 
covers only a part of the 
range of parameters. First of all,  it concerns only sparse graphs (in which edge probability tends to zero). Secondly,   for a significant range of the parameters $n,   m,   p$,   the obtained upper bound does not tend to zero,   while asymptotic normality is expected.
In particular, this concerns the case when  $m \ge n^2$, which is counterintuitive, as the larger $m$ is, the more $\Gnmp$ resembles an Erd\H{o}s--R\'enyi graph, as discussed in   \cite{FillScheinSinger2000Equivalence}. Furthermore, a long line of research has shown that for  $m\gg n^3$ the graph $\Gnmp$ and Erd\H{o}s--R\'enyi graph with the edge probability close to $1-(1-p^2)^{m}$ are asymptotically equivalent \cite{FillScheinSinger2000Equivalence,Rybarczyk2011Equivalence,KimLeeNa2018Equivalence,
	BrennanBresterNagaraj2020Equivalence}.

The aim of this paper is to develop a new,   more efficient method for finding an upper bound on the distance between a standardised random variable counting the number of small subgraphs in $\mathcal{G}(n,  m,  p)$ and the standard normal distribution. We achieved this by developing and combining some probabilistic as well as combinatorial techniques.  The probabilistic ingredient of the approach builds on ideas from \cite{PSbej,  PSejp,  PSstoch},   where the Stein method was combined with Malliavin calculus to obtain normal approximation bounds for functionals of independent random variables distributed uniformly on the interval $(-1,  1)$. These methods have been already used in the context of  small subgraphs in other random graph models \cite{MNS,   PSbej,  PSejp}; however,   they are not directly applicable to the $\mathcal{G}(n,  m,  p)$ model due to the aforementioned edge dependencies. For this reason, in Theorem \ref{thm:prob->discrete} we derive a new bound that applies not only to the subgraph count in $\Gnmp$, but to general $U$-statistics  of random vectors. 
On the other hand,  what is essential for establishing accurate estimates of the distance, 
is the estimation of the central moments of subgraphs existence indicators. We believe that omitting centralization was one of the reasons for the gap in the result from \cite{Dong2023}. We resolve this problem by  very subtle analysis of so called clique covers, that were introduced in \cite{1999KarSchSin} and describe  subgraphs existence conditions. Namely, we developed  new techniques for comparing clique covers of two not necessarily isomorphic graphs.    This allows us to estimate specific types of central moments that are crucial for the final results. Additionally, the probabilistic approach described above effectively limits the number of cases considered. 
Namely, analyzing the higher central moments of a random variable counting the number of edges in $\Gnmp$ poses a significant challenge and requires a thorough examination of  clique covers of certain subgraphs associated with these moments. Theorem \ref{thm:prob->discrete}  greatly simplifies this analysis, as it reduces the problem to examining a few specific edge configurations.
The method we present is not only applicable across the entire range of traditionally considered parameters $n,   m,   p$,   but it also provides tight  bounds on the convergence rate to the normal distribution, which correspond to the  necessary condition for the convergence. In a subsequent paper \cite{Cliques}, we develop new methods for the analysis of clique covers, which, combined with the tools developed in the present paper, allow us to obtain corresponding results for the number of cliques in random intersection graphs.

The paper is organized as follows. In Section~\ref{sec:proofProbabilistic}, we prove Theorem~\ref{thm:prob->discrete}. In Section~\ref{sec:IntegralsToGraphs}, we present a general framework for applying Theorem~\ref{thm:prob->discrete} to subgraph counts in $\Gnmp$.In Section~\ref{sec:preliminaries}, we establish preliminary results needed for the subsequent analysis of the number of edges. Finally, in the last tree sections, we prove Theorem~\ref{thm:Main} using the tools and results developed in the preceding sections.

In the article we   use the following asymptotic notation:
\begin{itemize}
	\item $a_n=o(b_n) \Leftrightarrow $ $a_n/b_n\to 0$  as $n\to \infty$,
	\item $a_n=O(b_n) \Leftrightarrow $ $|a_n|\le Cb_n$  for some $C>0$ and all  $n\in \N$,
	\item $a_n\ll b_n \Leftrightarrow   a_n=o(b_n)$ and $a_n, b_n>0$,
	\item $a_n\lesssim b_n \Leftrightarrow   a_n=O(b_n)$ and $a_n, b_n>0$,
	\item $a_n\asymp b_n \Leftrightarrow   a_n\lesssim b_n$ and $ b_n\lesssim a_n$.
\end{itemize}

We also use standard notation from graph theory.  For a graph $G=(V(G), E(G))$ we denote by $V(G)$ its vertex set and by $E(G)$ its edge set. Moreover $G\cup G'=(V(G)\cup V(G'), E(G)\cup E(G'))$ and $G\cap G'=(V(G)\cap V(G'), E(G)\cap E(G'))$. For $E'\in E(G)$ and $V'=\{v:\exists_{e\in E'}\, v\in e\}$, $G[E']=(V', E')$ is the graph induced on the edge set $E'$, and    $G-E'=(V(G),    E(G)\setminus E')$. We write $\cong$ for the graph isomorphism.

\section{Main results and discussion}
\label{sec:mainresults}

Let $m=m(n)\to \infty$ as $n\to \infty$ and $p=p(n)\in (0,  1)$.
In the random intersection graph $\Gnmp$,   we are  given  a set of vertices $\V$, $|\V|=n$, and a set of attributes $\A=\{a_1,  \ldots,  a_m\}$. Any  attribute $a_i$,   $1\le i\le m$,  is assigned to any vertex $v\in \V$  (we say that vertex $v$ {\sl chose} attribute $a_i$) independently with probability $p$. Then,   two vertices $v$ and $v'$ are connected by an edge in $\Gnmp$ whenever they chose at least one common attribute. The existence probability of any edge will be denoted by $\hatp$ and is given by
\begin{align}\label{eq:hatp}
	\hatp=1-(1-p^2)^{m}\asymp  1\wedge(mp^2).
\end{align}
Note that for $mp^2$ close to zero or infinity one can approximate $\hatp\sim  1-e^{-mp^2}$.
Furthermore,   for any vertex  $v_k$,   $k\in\{1,  ...,  n\}$,   we define a random vector $A^{(k)}$,   representing relations between $v_k$ and all the attributes,   by
$$A^{(k)}=\left(A^{(k)}_1,  \ldots,  A^{(k)}_m\right),  $$
where $A^{(k)}_i$ is the indicator of the event $\{v_k \text{ chose }a_i\}$. By definition $\{A^{(k)}_i: k\in\{1,  \ldots,  n\},   i\in\{1,  \ldots,  m\} \}$ creates  a family of independent random variables with a Bernoulli distribution $\mathcal B(1,   p)$. 
Hence,    the distribution $\mu_{m,   p}$ of  any $A^{(k)}$ is given by
$$\mu_{m,   p}(x)=p^{|x|}(1-p)^{m-|x|},  \ \ \ \ x\in\{0,  1\}^m,  $$
where 
$$|x|=\sum_{i=1}^mx_i,  $$
as in the multi-index notation. In particular,   we have
$$|A^{(k)}| \ \ \sim \ \ \mathcal B(m,  p).$$

Many of the  local statistics on $\Gnmp$ might be expressed in the form of a $U$-statistic
\begin{align}\label{eq:statistic}
	X=\sum_{1\le k_1<\ldots <k_r\le n}h(A^{(k_1)},  \ldots ,  A^{(k_r)}), 
\end{align}
where the function $h:{\left(\{0,  1\}^m\right)^{r}}\rightarrow\R$ is called a kernel. 
Our first goal is to examine the asymptotic normality of the standardised random variable 
$$\widetilde X:=\frac{X-\E[X]}{\sqrt{\Var[X]}}$$
by establishing bounds on its distance to a normally distributed random variable $\mathcal N\sim \mathcal N(0,  1)$. Such quantitative central limit theorems are usually described by means of the Kolmogorov distance or the Wasserstein one.  For two random variables $X$ and $Y$ we define the Kolmogorov distance $d_K(X,  Y)$ between them as
\begin{align*}
	d_{K}(X,  Y):=\sup_{t\in \R}\big|\P\left(X\le t\right)-\P(Y\le t)\big|,  
\end{align*}
while the Wasserstein distance $d_W(X,  Y)$ is given by
\begin{align*}
	d_{W}(X,  Y):=\sup_{\text{Lip}(h)\le1}\big|\E[h(X)-h(Y)]\big|,  
\end{align*}
where $\text{Lip}(h)$ stands for the   Lipschitz constant of a function $h:\R\rightarrow\R$. Since we  work with both of the distances  simultaneously,   we also denote
$$d_{K/W}(X,  Y):=\max\{d_{K}(X,  Y),   d_{W}(X,  Y)\}.$$
Let us remark that $d_{K}(X,  Y)\le1$ for any $X,  Y$ and $d_W(X,  Y)\le \E[|X-Y|]\le\sqrt{\E[X^2]}+\sqrt{\E[Y^2]}$.  Thus,  for standardised $X$ and $Y$ we have
\begin{align}\label{eq:d<2}
	d_{K/W}(X,  Y)\le2.
\end{align}

The study of the asymptotic normality of $U$-statistics already has a long history. However, we have not found any existing results that are fully suitable for the purposes of this article.
In particular, the bounds from \cite[Theorem 2.1]{Z} and \cite[Theorem 3.1]{CS} yield convergence rates that do not match the asymptotic normality conditions of $\widetilde \Xe$. The result of \cite[Theorem~3.3]{DP} requires the $U$-statistic to be degenerate (i.e., $\db h = h$ in the notation of \eqref{eq:dbh}), which is a strong limitation and does not apply to subgraph counts in $\Gnmp$, whereas Theorem 5.2 therein provides a bound for certain weaker distances induced by test functions from $\mathcal{C}^3(\R)$. A satisfactory bound for the Wasserstein distance could be derived by combining  \cite[Corollary 3.9]{D} and the proof of \cite[Proposition A.1]{PSstoch} and \cite[Theorem 4.1]{PSejp} are relatively close to satisfying our needs; however, the former concerns only the Wasserstein distance, while the latter requires the Hoeffding decomposition of the random variable. Additionally, both lead to rather involved expressions. Furthermore, they must be modified to be suitable for working with $U$-statistics of random vectors; otherwise, the constant in the distance bound would depend on $m$.
For these reasons, in Theorem \ref{thm:prob->discrete} we establish a new bound that is suitable for symmetric $U$-statistics on $\Gnmp$, including small subgraph counts.

For a probabilistic measure $\nu$, functions $f:{\left(\R^m\right)^{k}}\rightarrow\R$,   $g:{\left(\R^m\right)^{l}}\rightarrow\R$ and $0\le{a\le b}\le k\wedge l$ we define a contraction
\begin{align*}
	& f\ast_{b}^ag(x_1,  \ldots ,   x_{b-a},  y_1,  \ldots,  y_{k-b},  z_1,  \ldots ,   z_{l-b})
	\\\nonumber
	& :=
	\int_{{\left(\R^m\right)^a}} f(w_1,  \ldots,  w_a,   x_1,  \ldots ,   x_{b-a},  y_1,  \ldots,  y_{k-b}) \\\nonumber
	&\ \ \ \ \ \ \ \ \ \ \ \ \ \ \ \ \times g(w_1,  \ldots ,   w_a,  x_1,  \ldots ,   x_{b-a},  z_1,  \ldots ,   z_{l-b})d\nu^{\otimes ma}(w).
\end{align*}
This type of notation was first introduced in \cite{Kabanov} and has since become widely adopted.
Furthermore,   by
\begin{align}\label{eq:bar}
	\overline h(x):=h(x)-\int_{\left(\R^m\right)^{k}}h(y)d\nu^{\otimes mk}(y),  \ \ \ x\in \left(\R^m\right)^{k}
\end{align}
we denote a  centralization of the function $h$. Additionally, in order to simplify notation,  we  skip the indication of the measure space in the $L^p$-norms,  i.e.,  for any $f\in L^p\big((\R^m)^k,  \nu^{\otimes mk}\big)$,  $p\in[0, \infty]$,  $k\in\N$,  we write 
\begin{align}\label{eq:Lpnotation}
	\|f\|_p:=\|f\|_{L^p\left((\R^m)^k,  \nu^{\otimes mk}\right)}.
\end{align}
We are now in a position to state the first main result of the article.
\begin{theorem}\label{thm:prob->discrete}
	Let $Y_{k, l}$, $1\leq k\leq n$, $1\leq l\leq m$, be independent with distribution $\nu$, and denote $Y_k=\left(Y_{k,1}, \ldots, Y_{k, m}\right)$. For 
	\begin{align}\label{eq:statistic-nu}
		X=\sum_{1\le k_1<\ldots <k_r\le n}h(Y_{k_1},  \ldots ,  Y_{k_r})
	\end{align}
	with symmetric $h\in L^4\big((\R^m)^r,   \nu^{\otimes mr}\big)$, $r\geq1$, we have
	\begin{align*}
		&d_{K/W}\left(\frac{X-\E [X]}{\sqrt{\Var[X]}},  \mathcal N\right)\\
		&\le \frac{C_rn^{2r}}{\Var[X]}\Bigg({\sum_{0\le i<j\le r}}
		\frac{\| \overline h_j\ast_j^i  \overline h_j\|^2_2}{n^{3j-i}}+{\sum_{1\le i<j\le r}}\bigg(\frac{\| \overline h_j\ast_i^i  \overline h_j\|^2_2}{n^{2j}}
		+\frac{\| \overline h_j\ast_i^i  \overline h_i\|^2_2}{n^{j+i}}\bigg)\Bigg)^{1/2},  
	\end{align*}
	where $C_r$ is a  constant depending on $r$ only,   
	and
	\begin{align}\label{eq:defh_j}
		h_j(x_1,  \ldots ,  x_j)=\int_{\left(\{0,  1\}^m\right)^{r-j}}h(x_1,  \ldots ,  x_r)d\nu^{\otimes m(r-j)}(x_{j+1},  \ldots,   x_r).
	\end{align}
\end{theorem}

Let us note that one may strengthen the assertion of Theorem  \ref{thm:prob->discrete} by replacing the functions  $\overline h_j$ by their coordinate-wise centralisations $\db h_j$ (see \eqref{eq:dbh}),   which is presented in Theorem \ref{thm:generaldb}.  Nevertheless,   in many cases such a change  significantly increases complexity of calculations and not necessarily improves the order of the bound.  In particular, this is the case of the edge count in $\Gnmp$, while it might be shown that considering coordinate-wise centralisation does not lead to better distance bounds. On the other hand,   one can also waive any centralisation of $h_j$'s,   however,   in that case some additional  terms may appear,  that are related to the last bound in Lemma \ref{lem:db<}. Furthermore,   modifying the proof one can generalize Theorem  \ref{thm:prob->discrete} by considering kernels $h$ depending on vertices $v_{k_1},  \ldots,   v_{k_r}$. We decided against  this idea in order to keep the results simple and easy to apply.

As it is often the case with results of this type,  it is unclear whether the bounds given in Theorem \ref{thm:prob->discrete} are tight. As a basic test, we take $m=r=1$ and $h(x)=x$, which corresponds to the classical CLT. Under the assumptions that $\Var[Y_{1,1}]=1$ and $\E[Y_{1,1}^4]<\infty$, this yields the optimal rate $n^{-1/2}$.  In the remainder of the article, we will show a less trivial example demonstrating that this estimate is sharp.

Namely, in Section~\ref{sec:IntegralsToGraphs}, we explain how the theorem can be applied to the standardized number of subgraphs in the random intersection graph model $\Gnmp$, with particular emphasis on cliques $K_r$ on $r$ vertices. In this setting, in Theorem~\ref{thm:prob->discrete}, we take $\nu \sim \mathcal{B}(1,p)$ and $\nu^{\otimes m} = \mu_{m,p}$. We then demonstrate in detail, using the number of edges as an example, how Theorem~\ref{thm:prob->discrete} can be applied, and prove the following theorem.


\begin{theorem}\label{thm:Main} 
	Let $\Xe$ be the random variable counting edges in $\Gnmp$ and let $\mathcal N\sim\mathcal N(0,  1)$ be a normally distributed random variable. If $m=m(n)\rightarrow \infty $ as $n\rightarrow\infty$ and   $p=p(n)\le 0.1$,  then we have
	$$\frac{ \Xe-\E[\Xe]}{\sqrt{\Var[\Xe]}}\stackrel{\mathcal D}{\longrightarrow}\mathcal N \ \ \ \Longleftrightarrow \ \ \ n^2\hatp(1-\hatp)\rightarrow\infty, $$
	and the convergence rate is bounded by
	\begin{align}\label{eq:mainbound1}
		d_{K/W}\left(\frac{ \Xe-\E[\Xe]}{\sqrt{\Var[\Xe]}},  \mathcal N\right)\le C \left(\frac1{n^2\hatp(1-\hatp)}\right)^{1/4}, \ \ \ \ n\ge n_0, m\ge m_0,
	\end{align}
	for some $C,  n_0, m_0>0$. Additionally,  if $mp^3\le1$,  the following  sharper bound holds true
	\begin{align}\label{eq:mainbound2}
		d_{K/W}\left(\frac{ \Xe-\E[\Xe]}{\sqrt{\Var[\Xe]}},  \mathcal N\right)\le C\left(\frac1{n^2\hatp(1-\hatp)}+\frac1n+\frac1{m}\right)^{1/2}.
	\end{align}
\end{theorem}
Applying Theorem~\ref{thm:prob->discrete} in full generality, however, is technically involved and requires a detailed analysis of so-called clique covers. We therefore defer the treatment of cliques $K_{r}$, $r\ge 3$, to a companion paper \cite{Cliques}, where, due to the complexity of the analysis, the corresponding result will be established for a somewhat more restrictive range of parameters.

The quantity $n^2\hatp(1-\hatp)$ is comparable to $\E[\Xe]\wedge \E\[{n\choose 2}-\Xe\]$,   where ${n\choose 2}-\Xe$ might be interpreted as the number of non-edges (edges that do not exist). An analogous expression appears in the case of the $\Gnp$ model,   and the condition translates into $a)$ we need a lot of edges so that the asymptotic distribution is continuous and $b)$ we need a lot of non-edges so that the random graph is not to close to the (deterministic) complete graph. These conditions are therefore rather weak and expected. Furthermore,  they are necessary  for any $p=p(n)$,  and not only for $p\le0.1$ (see Lemma \ref{lem:necessary}).  On the other hand,   veryfying  their sufficiency turns out  to be a particularly  challenging task.   This is a fundamental  difference between the $\Gnmp$ and $\Gnp$ models,   as in the $\Gnp$ model the distribution of the  number of edges is simply  the binomial one $\mathcal B\big({n\choose 2},  p\big)$,  and hence some classical probabilistic results apply.

Theorem \ref{thm:Main}  covers and extends the range of parameters considered usually in the literature such as $mp^2\rightarrow0$ or $mp^3\rightarrow0$. Nevertheless,  it is still natural to ask what happens for $p>0.1$. Apparently,  this case is quite complex,  so we present only some partial results. First of all,  in Proposition \ref{prop:K14} we derive the bound
\begin{align}\label{aux10}
	d_{K/W}(\widetilde \Xe ,   \mathcal{N})
	\lesssim \left(\frac1{n^2\hatp(1-\hatp)}+\frac1n+\frac{n^{5}\, {\P\left(K_{1, 4}\subseteq \cGnmp\right)}}{\Var^2 [\Xe]}\right)^{1/2}, 
\end{align}
valid whenever $mp^3\geq1$,  where $K_{1, 4}$ is a star with four leaves on some given vertices (note that the event $\{K_{1, 4}\subseteq \cGnmp\}$ is often understood as existence of at least one copy of $K_{1, 4}$ in  $\cGnmp$,  which is not the case). The last term in \eqref{aux10} diverges to infinity for $p$ close enough to $1$,  even if  $n^2\hatp(1-\hatp)\rightarrow\infty$ and $\E[\Xe]^2/\Var[\Xe]\rightarrow\infty$,  where the latter condition is a necessary and sufficient condition for the asymptotic normality  in the $\Gnp$ model \cite{Ruc},  and might be proven as a necessary condition in $\Gnmp$ as well,   see \cite[proof of Theorem 3.5]{MNS}. Furthermore,  the same term turns out to determine the order of  $\E[\widetilde \Xe^4]$ for $p\rightarrow1$, which might be shown using formula \eqref{eq:V}. Let us note that a similar phenomenon occurs in the classical $\Gnp$ model  \cite[Theorem 3.6]{MNS}, where, however, the divergence can be effectively mitigated by a negative power of $n$. In contrast, in the $\Gnmp$ model, the common assumption $m = n^{\alpha}$ with $\alpha > 0$ is insufficient to suppress the leading term of order $(1-p)^{-m}$. Thus,  the fourth moment of $\widetilde \Xe$ blows up,  which may suggest that asymptotic normality of $\widetilde \Xe$ does not hold then.  In general,  boundedness of the fourth moment is not required for asymptotic normality,  however,  one can observe its convergence to $\E[\mathcal N^4]=3$ in many contexts,  related to random graphs or not \cite{BDM,  MNS,  dJ,  NP}.

Regarding intermediate values of $p$,  in Proposition \ref{prop:mlnn} we show existence of some kind of threshold function. Precisely,  for $\varepsilon\in(0,  \tfrac12)$ and  $p=p(n)\in\left(\varepsilon,  1-\varepsilon\right)$ it holds that 
$$
\begin{array}{ll}
	\widetilde \Xe\stackrel{\mathcal D}{\rightarrow}\mathcal N, & \text{ if }\frac{m}{\ln n}\rightarrow0,  \\
	\widetilde \Xe\stackrel{\mathcal D}{\not\rightarrow}\mathcal N, & \text{ if }\frac{m}{\ln n}\rightarrow\infty.
\end{array}
$$
The latter condition is already known,  since for $p=\sqrt{
	(2 \ln n + \omega)/m}$,  where $\omega\rightarrow\infty$,  with high probability $\Gnmp$ is complete \cite{1999KarSchSin}. The behaviour of the distribution of $\Xe$ for $m\asymp  \ln n$ requires a deep analysis,  and we leave it here as an open problem.

\section{Normal approximation of $U$-statistics} \label{sec:proofProbabilistic}

The proof of Theorem \ref{thm:prob->discrete} builds on the results established in \cite{PSejp0, PSejp, PSstoch} concerning functionals of independent uniformly distributed random variables. The methods employed in those works are based on Malliavin calculus on a specially constructed product space.  In particular, it contains the proofs of Theorem \ref{thm:generaldb} and Lemma \ref{lem:db<}. Below, we present the  main steps of the proof of Theorem \ref{thm:prob->discrete}, that might be of separate interest as well.

Denoting the operator
$$\Phi_{x_i}h(x):=h(x)-\int_{{ \R^m}}h(x)d\nu^{\otimes m} (x_i),  $$
we define the coordinate-wise  centralization $\db h$ of the function $h:( \R^m)^k\rightarrow\R$  by
\begin{align}\label{eq:dbh}
	\db h(x)=\Phi_{x_1}\ldots \Phi_{x_k}h(x),  \ \ \ \ x\in( \R^m)^k.
\end{align}
This kind of centralization is much stronger than the classical one given  in \eqref{eq:bar},   as any  integral of $\db h$ over $ \R^m$ with respect to any coordinate vanishes almost surely,   while in the case of $\overline h$ only the integral over the whole space $( \R^m)^k$ is required to be zero. The $U$-statistics whose kernels satisfy $h=\db h$ are called degenerate \cite{GGG}.

\begin{theorem} \label{thm:generaldb}
	Let $X$ be of the form \eqref{eq:statistic-nu} 
	with symmetric $h\in L^4\big(( \R^m)^r,  \nu ^{\otimes m  r}\big)$. Then we have
	\begin{align}\label{eq1:prob->discrete}
		&d_{K/W}\left(\frac{X-\E [X]}{\sqrt{\Var[X]}},  \mathcal N\right)\\
		&\le \frac{C_rn^{2r}}{\Var[X]}\Bigg({\sum_{0\le i<j\le r}}
		\frac{\|\db h_j\ast_j^i\db h_j\|^2_2}{n^{3j-i}}+{\sum_{1\le i<j\le r}}\bigg(\frac{\|\db h_j\ast_i^i\db h_j\|^2_2}{n^{2j}}
		+\frac{\|\db h_j\ast_i^i\db h_i\|^2_2}{n^{j+i}}\bigg)\Bigg)^{1/2},\nonumber
	\end{align}
	where $C_r>0$ is a constant depending on $r$ only,   and the functions $h_j$ are defined in 
	\eqref{eq:defh_j}.
\end{theorem}

\begin{proof}
	Let $(U_k)_{k\geq 1}$ denote the i.i.d. sequence of uniformly distributed
	$(-1,  1)$-valued random variables. Given $f_N$ in the space $\widehat{L}^2(\R_+^N)$ of 
	square integrable symmetric functions on $\R_+^N:=[0,  \infty)^N$ that
	vanish outside of (some kind of diagonal) 
	$$
	\Delta_N : = \bigcup_{
		\substack{k_1, \ldots,  k_N\in\N\\k_i\neq k_j\text{ for }i\neq j}
	}
	[2k_1-2,  2k_1]\times \cdots \times [2k_N-2,  2k_N],   
	$$
	we define the multiple stochastic integral
	$$ 
	I_N(f_N)
	= 
	N!\int_0^\infty \int_0^{t_N}\cdots \int_0^{t_2}f_N(t_1,  \ldots ,  t_N)d(Y_{t_1}-t_1/2)\cdots d(Y_{t_N}-t_N/2),   
	$$
	with respect to the jump process 
	$\displaystyle Y_t : =\sum_{k=1}^\infty\mathbf1_{[2k-1+U_k,  \infty)}(t)$,  
	$t\in \real_+$. If $f_N\in \widehat{L}^2(\R_+^N)$ satisfies the condition
	\begin{align}\label{ass:int0}
		\int_{2k-2}^{2k}f_N(x)dx_1=0
	\end{align}
	for any $k\in \N$,   the multiple stochastic integral takes the following  form
	\begin{align}
		\nonumber
		I_N(f_N)&=N!\int_0^\infty \int_0^{t_N}\cdots \int_0^{t_2}f_N(t_1,  \ldots ,  t_N)dY_{t_1}\cdots dY_{t_N}\\[6pt]\label{eq:I=sum}
		&=\sum_{\substack{k_1, \ldots,  k_N\in\N\\k_i\neq k_j\text{ for }i\neq j} } f_N (2k_1-1+U_{k_1},  \ldots ,  2k_N-1+U_{k_N}). 
	\end{align}
	Proposition A.3 in \cite{PSstoch} states that  if $f$ does not satisfy \eqref{ass:int0},   there is a function $\hat f_N$ satisfying \eqref{ass:int0} such that $I_N(f_N)= I_N(\hat f_N)$,  which  is given by
	\begin{align}
		\label{eq:barf}
		\hat{f}_N(t_1,  \ldots ,  t_N)
		= 
		\Psi_{t_1} \cdots \Psi_{t_N} f_N(t_1,  \ldots ,  t_N),  
	\end{align}
	where
	\begin{align*}
		\Psi_{t_i}f_N(t_1,  \ldots ,   t_N):=f_N(t_1,  \ldots ,   t_N)-
		\frac{1}{2}
		\int_{2\lfloor t_i/2\rfloor}^{2\lfloor t_i/2\rfloor+2}f(t_1,  \ldots ,  t_{i-1},  s,  t_{i+1},  \ldots ,  t_N)ds,  
	\end{align*}
	$i=1,  \ldots ,   N$,  
	$t_1,  \ldots ,  t_N \in \real_+$,  is a centralization of the function $f$ with respect to $t_i$ on the interval $({2\lfloor t_i/2\rfloor},  {2\lfloor t_i/2\rfloor+2})$.

	The formula \eqref{eq:I=sum} represents a general $U$-statistic and appears to be very convenient in expressing counting statistics in random graphs. However,  the function $f_N$ usually does not satisfy \eqref{ass:int0},  and then the whole sum does not coincide with any  multiple stochastic integral. Thus,  we express it as a sum of the multiple stochastic integral,  which is called the chaos  decomposition or the Hoeffding decomposition.   Namely,  by Proposition 3.1 in \cite{PSstoch} we have
	\begin{align}\label{eq:sum=sumI}
		\sum_{ \substack{k_1, \ldots,  k_N\in\N\\k_i\neq k_j\text{ for }i\neq j} } f_N (2k_1-1+U_{k_1},  \ldots ,   2k_N-1+U_{k_N})=\sum_{k=0}^NI_k(f_k),  
	\end{align}
	where
	\begin{align*}
		f_k(x_1,  \ldots ,  x_k)=\frac{1}{2^{N-k}}
		{{N}\choose{k}} \int_{(0,  \infty)^{N-j}} f_N(x_1,  \ldots ,  x_N)dx_{N-j+1}\ldots dx_{N}.
	\end{align*}
	Eventually,   we present normal approximation bounds for this kind of random variables. 
	Let $X\in L^2 (\Omega )$ be written as a sum 
	\begin{align}\label{eq:X}
		X=\sum_{i=1}^d I_i(f_i)
	\end{align}
	of multiple stochastic integrals,  where  every
	$f_i\in \widehat{L}^2(\R_+^i)$ 
	satisfies \eqref{ass:int0}. Then,   Proposition 3.2  from  \cite{PSejp} gives us
	\begin{align}\label{dWK<sum}  
		d_{K/W}(X ,   \mathcal{N})
		\le   
		C_d &\Bigg(
		\sum_{0\le i< j\le d}\left\|  f_{j}\star_{j}^i  f_{j}\right\|^2_{
			L^2(\R_+^{j-i})}\\
		&\quad+\sum_{1\le i< j\le d}
		\left(\left\|  f_{j}\star_{i}^i  f_{j}\right\|^2_{
			L^2(\R_+^{2(j-i)})}+\left\|  f_{j}\star_{i}^i  f_{i}\right\|^2_{
			L^2(\R_+^{j-i})}\right)\Bigg)^{1/2}, \nonumber
	\end{align}
	for some $C_d>0$,   where
	\begin{align*}
		&f_i\star_k^lf_j(x_1,  \ldots,   x_{k-l},   y_1,  \ldots,   y_{i-k},   z_1,  \ldots,   z_{j-k})\\
		&=2^{-l}\int_{\R^l} f(w_1,  \ldots,  w_l,   x_1,  \ldots ,   x_{k-l},  y_1,  \ldots,  y_{i-k}) \\\nonumber
		&\ \ \ \ \ \ \ \ \ \ \ \ \ \ \times g(w_1,  \ldots ,   w_l,  x_1,  \ldots ,   x_{k-l},  z_1,  \ldots ,   z_{j-k})dw_1\ldots dw_l.
	\end{align*}

	Since our main tool is based on the sequence of uniformly distributed random variables,   we will now show how to transform them into  random vectors $Y_k$'s. Unfortunately,   we cannot treat all the coordinates of all of the vectors $Y_k$'s as one sequence of independent random variables,   since then the summation index $d$ in \eqref{dWK<sum} would depend on $m$,   and so would  the constant in the distance bound.

	Let $U\sim U(0,  1)$ be uniformly distributed on $(0, 1)$ and $Z$ be any random variable.  It is well known that $F_Z^{-1}(U)\stackrel{d}{=}Z$,  where $F_Z^{-1}$ is the generalised inverse function of the distribution function $F_Z(t)=\P\left(Z\le t\right)$ of $Z$. An analogous property holds true for $Z$ being a random vector $Z=\left(Z_1,  \ldots ,  Z_m\right)$ with independent coordinates. Namely,  let $U_i$,  $i\in\{1,  \ldots,  m\}$  be a number constructed from every $m$ digits of the decimal expansion of $U$,  starting from the position $i$,  i.e., 
	$$U_i:=\sum_{s=0}^\infty \frac{\lfloor 10^{ms+i}U\rfloor-10\lfloor 10^{ms+i-1}U\rfloor}{10^{ms+1}}.$$ Then $\{U_i\}_{1\le i\le m}$ creates a family of independent random variables uniformly distributed on $(0,  1)$ and hence
	$$Z\stackrel{d}{=}\left(F_{X_1}^{-1}(U_1),  \ldots ,  F_{X_m}^{-1}(U_m)\right). $$
	Thus, taking $Z=Y_k$,  there exists a function $a:(0, 1)\rightarrow ( \R)^m$ satisfying
	\begin{align}\label{eq:a=A}
		a(U)\stackrel{d}{=}Y_k,  \ \ \ \ 1\le k\le n.
	\end{align}
	Consequently,   since the random variables $\tfrac12(U_k+1)$,   $k\geq1$,   are uniformly distributed on $(0,  1)$,    we obtain
	\begin{align*}
		X&\stackrel{d}{=}\sum_{1\le k_1<\ldots <k_r\le n}h\left(a\big(\tfrac12(U_{k_1}+1)\big),  \ldots ,  a\big(\tfrac12(U_{k_r}+1)\big)\right)\\
		&\stackrel{\phantom{d}}{=}\sum_{\substack{k_1, \ldots,  k_n\in\N\\k_i\neq k_j\text{ for }i\neq j}}f \left(2k_1+1+U_{k_1},  \ldots ,  2k_r+1+U_{k_r}\right),  
	\end{align*}
	where
	$$f(x)=\frac1{r!}\mathbf1_{[0,  2n]^r}(x)h\big(a\left(\tfrac12(x_1-2\lfloor \tfrac{x_1}2\rfloor)\right),  \ldots ,  a\left(\tfrac12(x_r-2\lfloor \tfrac{x_r}2\rfloor)\right)\big).$$
	Then,   \eqref{eq:barf} and \eqref{eq:sum=sumI} give as 
	\begin{align*}
		X&\stackrel{d}{=}\sum_{j=0}^rI_j(f_j)=\sum_{j=0}^rI_j(\hat f_j),  
	\end{align*}
	where 
	\begin{align*}
		&f_j(x_1,  \ldots ,  x_j)\\
		&=\frac{1}{r! 2^{r-j}}
		{{r}\choose{j}} \int_{(0,  \infty)^{r-j}} f(x_1,  \ldots ,  x_j,  y_1,  \ldots ,  y_{r-j})dy_1\ldots dy_{r-j}\\
		&=\frac{1}{r! }
		{{r}\choose{j}}\\ &\sum_{\substack{1\le i_1, ..., i_{r-j}\le n\\|\{i_1, ..., i_{r-j}, \lceil \frac{x_1}{2}\rceil, ..., \lceil \frac{x_j}{2}\rceil \}|=r\\}}\int_{2i_1-2}^{2i_1} \ldots \int_{2i_{r-j}-2}^{ 2i_{r-j}} f(x_1,  \ldots ,  x_j,  y_1,  \ldots ,  y_{r-j})\frac{dy_1\ldots dy_{r-j}}{2^{r-j}}\\
		&=\frac{1}{r!}
		{{r}\choose{j}} {(n-j)_{r-j}}\\
		&\times\int_{(0,  2)^{r-j}}h\left(a\left(\tfrac12(x_1-2\lfloor \tfrac{x_1}2\rfloor)\right),  \ldots ,  a\left(\tfrac12(x_j-2\lfloor \tfrac{x_j}2\rfloor,  a(z_1),  \ldots ,  a(z_{r-j}))\right)\right)\\&\hspace{9cm}\frac{dz_1}2\ldots \frac{dz_{r-j}}2\\
		&=\frac{(n-j)_{r-j}}{j!(r-j)!}\\
		&\times\int_{\left(\{0,  1)\}^m\right)^{r-j}} h\left(a\left(\tfrac12(x_1-2\lfloor \tfrac{x_1}2\rfloor)\right),  \ldots ,  a\left(\tfrac12(x_j-2\lfloor \tfrac{x_j}2\rfloor,  w_1,  \ldots ,  w_{r-j})\right)\right)\\[-5pt]
		&\hspace{90mm}d\nu (w_1)\ldots d\nu (w_{r-j})\\
		&=\frac{(n-j)_{r-j}}{j!(r-j)!} h_j\left(a\left(\tfrac12(x_1-2\lfloor \tfrac{x_1}2\rfloor)\right),  \ldots ,  a\left(\tfrac12(x_j-2\lfloor \tfrac{x_j}2\rfloor\right)\right),  
	\end{align*}
	where $h_j$'s are introduced in \eqref{eq:defh_j}. In particular,  due to the symmetry of $h$,   every $\hat f_j$ might be expressed by $\db h_j$ in the following manner
	\begin{align}\label{eq:f=h}
		\hat f_j(x_1,  \ldots ,  x_j)=\frac{(n-j)_{r-j}}{j!(r-j)!} \db h_j\left(a\left(\tfrac12(x_1-2\lfloor \tfrac{x_1}2\rfloor)\right),  \ldots ,  a\left(\tfrac12(x_j-2\lfloor \tfrac{x_j}2\rfloor\right)\right).
	\end{align}
	Furthermore,   since $I_0(f_0)=\E[X]$,   we arrive at
	$$\widetilde X=\frac{X-\E[X]}{\sqrt{\Var[X]}}=\sum_{j=1}^rI_j\left(\frac{\hat f_j}{\sqrt{\Var[X]}}\right).$$
	Thus,   by  \eqref{dWK<sum} we get
	\begin{align*}
		d_{K/W}(\widetilde X ,   \mathcal{N})
		\le   
		\frac{C_r}{\Var[X]}&\Bigg(
		\sum_{0\le i< j\le n}\left\| \hat f_{j}\star_{j}^i  \hat f_{j}\right\|^2_{
			L^2(\R_+^{j-i})}\\
		&+\sum_{1\le i< j\le n}
		\left(\left\|  \hat f_{j}\star_{i}^i  \hat f_{j}\right\|^2_{
			L^2(\R_+^{2(j-i)})}+\left\|  \hat f_{j}\star_{i}^i  \hat f_{i}\right\|^2_{
			L^2(\R_+^{j-i})}\right)\Bigg)^{1/2}.
	\end{align*}
	The last step is to   estimate the $L^2$ norms above by means of the functions $\db h_i$. Namely,   by \eqref{eq:f=h}  and \eqref{eq:a=A} we have 
	\begin{align*}
		&\left\| \hat f_{j}\star_{j}^i \hat f_{j}\right\|^2_{L^2(\R_+^{j-i})}=\\
		&=\int_{(0,  \infty)^{j-i}}\left(\int_{(0,  \infty)^{i}}\hat f^{\, 2}_j(x_1,  \ldots ,  x_{j})\frac{dx_1\ldots dx_{i}}{2^i}\right)^2dx_{i+1}\ldots dx_j\\
		&\le  n^{j-i}\int_{(0,  2)^{j-i}}\left(n^i\int_{(0,  2)^i}\Big(n^{r-j}\db h_j\left(a\left(\tfrac12(x_1)\right),  \ldots ,  a\left(\tfrac12(x_j\right)\right)\Big)^2\frac{dx_1\ldots dx_{i}}{2^i}\right)^2\\
		&\hspace{10cm}dx_{i+1}\ldots dx_j\\[3pt]
		&= 2^{j-i}n^{4r-3j+i}\|\db h_j\ast_j^i\db h_j\|^2_2,  
	\end{align*}
	as well as 
	\begin{align*}
		&\left\| \hat f_{j}\star_{i}^i \hat f_{j}\right\|^2_{L^2(\R_+^{2(j-i)})}=\\
		&=\int_{(0,  \infty)^{j-i}}\int_{(0,  \infty)^{j-i}}\left(\int_{(0,  \infty)^{i}}\hat f_j(x_1,  \ldots ,  x_{j})\hat f_j(x_1,  \ldots ,  x_{i},  y_1,  \ldots ,  y_{j-i})\frac{dx_1\ldots dx_{i}}{2^i}\right)^2\\
		&\hspace{250pt}dx_{i+1}\ldots dx_jdy_1\ldots dy_{j-i}\\
		&\le   n^{2(j-i)}\int_{(0,  2)^{j-i}}\int_{(0,  2)^{j-i}}\bigg(n^i\int_{(0,  2)^i}\Big(n^{r-j}\db h_j\left(a\left(\tfrac12(x_1)\right),  \ldots ,  a\left(\tfrac12(x_j\right)\right)\Big)\\
		&\hspace{10pt}\times\Big(n^{r-j}\db h_j\left(a\left(\tfrac12(x_1)\right),  \ldots ,  a\left(\tfrac12(x_i\right),  a\left(\tfrac12(y_1)\right),  \ldots ,  a\left(\tfrac12(y_{j-i}\right)\right)\Big)\frac{dx_1\ldots dx_{i}}{2^i}\bigg)^2\\
		&\hspace{250pt}dx_{i+1}\ldots dx_jdy_1\ldots dy_{j-i}\\
		&=2^{2(j-i)}n^{4r-2j}\|\db h_j\ast_i^i\db h_j\|^2_2,  
	\end{align*}
	and eventually
	\begin{align*}
		&\left\| \hat f_{j}\star_{i}^i\hat  f_{i}\right\|^2_{L^2(\R_+^{j-i})}=\\
		&=2^{-2i}\int_{(0,  \infty)^{j-i}}\left(\int_{(0,  \infty)^{i}}\hat f_j(x_1,  \ldots ,  x_{j})\hat f_i(x_1,  \ldots ,  x_i)\frac{dx_1\ldots dx_{i}}{2^i}\right)^2dx_{i+1}\ldots dx_j\\
		&\le  n^{j-i}\int_{(0,  2)^{j-i}}\bigg(n^i\int_{(0,  2)^i}\Big(n^{r-j}\db h_j\left(a\left(\tfrac12(x_1)\right),  \ldots ,  a\left(\tfrac12(x_j\right)\right)\Big)\\
		&\hspace{80pt}\times\Big(n^{r-i}\db h_i\left(a\left(\tfrac12(x_1)\right),  \ldots ,  a\left(\tfrac12(x_i\right)\right)\Big)\frac{dx_1\ldots dx_{i}}{2^i}\bigg)^2dx_{i+1}\ldots dx_j\\
		&= 2^{j-i}n^{4r-j-i}\|\db h_j\ast_i^i\db h_i\|^2_2,  
	\end{align*}
	which completes the proof.
\end{proof}

Next,   we will show how the norms of the form $\|\db h_i\ast_k^l\db h_j\|_2$ appearing in  Theorem~\ref{thm:generaldb} might be bounded by means of their counterparts without centralization.  Additionally,   the last inequality below indicates how the bound in Theorem \ref{thm:prob->discrete} would change if  one required them to be expressed by means of functions $h_j$ instead of their centralised versions  $\overline h_j$. 
\begin{lemma}\label{lem:db<} For symmetric $h\in L^4\big(( \R^m)^r,  \nu ^{\otimes m  r}\big)$ and $0\le i< j\le r$,  it holds that 
	\begin{align*}
		\|\db h_j\ast_j^i\db h_j\|_2&\lesssim \| h_j\ast_j^i h_j\|_2,  \\
		\|\db h_j\ast_i^i\db h_j\|_2&\lesssim \| h_j\ast_i^i h_j\|_2,  \\
		\|\db h_j\ast_i^i\db h_i\|_2
		&\lesssim  \|h_{j-i}\|_2\left|\int_{( \R^m)^r}h(x)d\nu ^{\otimes m  r}(x)\right|+\sum_{s=1}^i\|h_{j-i+s}\ast_s^s h_s\|_2
		.
	\end{align*}
\end{lemma}

\begin{proof}
	By the standard bound for variance 
	\begin{align}
		\label{eq:Holder}
		\int_{ \R^m}\big(\Phi_{t}f(t)\big)^{2} d\nu^{\otimes m} (t)\le \int_{ \R^m}\big(f(t)\big)^{2} d\nu^{\otimes m} (t),  \ \ \ \ f\in L^2\left( \R^m,   \nu^{\otimes m} \right),  
	\end{align}
	we get
	\begin{align}\nonumber
		&\|\db h_j\ast_j^i\db h_j\|^2_2\\\nonumber
		&=\int_{( \R^m)^{j-i}}\left(\int_{( \R^m)^i}\left(\Phi_{x_1}\ldots \Phi_{x_j}h_j(x)\right)^2d\nu ^{\otimes m  i}(x_1,  \ldots,   x_i)\right)^2d\nu ^{\otimes m  (j-i)}(x_{i+1},  \ldots,   x_j)\\\label{aux1}
		&\le \int_{( \R^m)^{j-i}}\left(\int_{( \R^m)^i}\left(\Phi_{x_{i+1}}\ldots \Phi_{x_j}h_j(x)\right)^2d\nu ^{\otimes m  i}(x_1,  \ldots,   x_i)\right)^2\\
		\nonumber 
		&\hspace{250pt}d\nu ^{\otimes m  (j-i)}(x_{i+1},  \ldots,   x_j).
	\end{align}
	Furthermore,   we observe that symmetry of $h$ gives us
	\begin{align}\label{aux2}
		\int_{ \R^m}h_j(x_1,  \ldots,  x_j)d\nu (x_j)=h_{j-1}(x_1,  \ldots,   x_{j-1}),  
	\end{align} 
	and hence 
	\begin{align*}
		\Phi_{x_{i+1}}\ldots \Phi_{x_j}h_j(x)&=\sum_{s=0}^{j-i}(-1)^{j-i-s}\sum_{i+1\le k_1<\ldots<k_{s}\le j}h_{i+s}(x_1,  \ldots,   x_i,   x_{k_1},  \ldots,   x_{k_{s}}).
	\end{align*}
	Since also,   by Cauchy-Schwarz inequality,  
	\begin{align*}
		\left|h_{i+s}(x_{1},  \ldots,   x_{i+s})\right|&=\left|\int_{( \R^m)^{s}}h_j(x_1,  \ldots,   x_j)d\nu ^{\otimes m  s}(x_{i+s+1},  \ldots,   x_j)\right|\\
		&\le \left(\int_{( \R^m)^{s}}\big(h_j(x_1,  \ldots,   x_j)\big)^2d\nu ^{\otimes m  s}(x_{i+s+1},  \ldots,   x_j)\right)^{1/2}\\
		&=\Big((h_j\ast_j^{j-i-s}h_j)(x_{1},  \ldots,   x_{i+s})\Big)^{1/2},  
	\end{align*}
	and 
	$$\| h_j\ast_j^{j-s} h_j\|^2_2\le\| h_j\ast_j^{i} h_j\|^2_2, \ \ \ 0\le i\le j-s, $$
	we get from \eqref{aux1}
	\begin{align*}
		&\|\db h_j\ast_j^i\db h_j\|^2_2\\
		&\lesssim \sum_{s=0}^{j-i} \int_{( \R^m)^{j-i}}\left(\int_{( \R^m)^i}\left(h_{i+s}(x_1,  \ldots,   x_{i+s})\right)^2d\nu ^{\otimes m  i}(x_1,  \ldots,   x_i)\right)^2\\
		&\hspace{9cm}d\nu ^{\otimes m  (j-i)}(x_{i+1},  \ldots,   x_j)\\
		&\lesssim \sum_{s=0}^{j-i} \int_{( \R^m)^{j-i}}\left(\int_{( \R^m)^i}(h_j\ast_j^{j-i-s}h_j)(x_{1},  \ldots,   x_{i+s})d\nu ^{\otimes m  i}(x_1,  \ldots,   x_i)\right)^2\\
		&\hspace{9cm}d\nu ^{\otimes m  (j-i)}(x_{i+1},  \ldots,   x_j)\\
		&=\sum_{s=0}^{j-i}\| h_j\ast_j^{j-s} h_j\|^2_2\le \sum_{s=0}^{j-i}\| h_j\ast_j^{i} h_j\|^2_2=(j-i)\| h_j\ast_j^{i} h_j\|^2_2, 
	\end{align*}
	which proves the first inequality of the lemma.
	
	We now turn our attention to the middle inequality from the assertion.  Using \eqref{eq:Holder} we bound
	\begin{align*}\nonumber
		&\|\db h_j\ast_i^i\db h_j\|^2_{2}\\\nonumber
		&=\int_{( \R^m)^{j-i}}\int_{( \R^m)^{j-i}}
		\bigg(\int_{( \R^m)^i}\left(\Phi_{x_1}\ldots \Phi_{x_i}\Phi_{y_1}\ldots \Phi_{y_{j-i}}h_j(x,  y)\right)\\\nonumber
		&\hspace{55pt}\left(\Phi_{x_1}\ldots \Phi_{x_i}\Phi_{z_1}\ldots \Phi_{z_{j-i}}h_j(x,  z)\right)d\nu ^{\otimes m  i}(x)\bigg)^2d\nu ^{\otimes m  (j-i)}(y)d\nu ^{\otimes m  (j-i)}(z)\\\nonumber
		&=\int_{( \R^m)^{j-i}}\int_{( \R^m)^{j-i}}
		\bigg(\Phi_{y_1}\ldots \Phi_{y_{j-i}}\Phi_{z_1}\ldots \Phi_{z_{j-i}}\int_{( \R^m)^i}\left(\Phi_{x_1}\ldots \Phi_{x_i}h_j(x,  y)\right)\\\nonumber
		&\hspace{110pt}\left(\Phi_{x_1}\ldots \Phi_{x_i}h_j(x,  z)\right)d\nu ^{\otimes m  i}(x)\bigg)^2d\nu ^{\otimes m  (j-i)}(y)d\nu ^{\otimes m  (j-i)}(z)\\\nonumber
		&\le\int_{( \R^m)^{j-i}}\int_{( \R^m)^{j-i}}
		\bigg(\int_{( \R^m)^i}\left(\Phi_{x_1}\ldots \Phi_{x_i}h_j(x,  y)\right)\left(\Phi_{x_1}\ldots \Phi_{x_i}h_j(x,  z)\right)d\nu ^{\otimes m  i}(x)\bigg)^2\\
		&\hspace{245pt}d\nu ^{\otimes m  (j-i)}(y)d\nu ^{\otimes m  (j-i)}(z).
	\end{align*}
	Next,  we represent the square of the integral as a multiplication of two integrals,   apply the Fubini theorem and take advantage of \eqref{eq:Holder} once again. Namely,   we have
	\begin{align*}
		&\|\db h_j\ast_i^i\db h_j\|^2_{{2}}\\\nonumber
		&\le\int_{( \R^m)^{j-i}}\int_{( \R^m)^{j-i}}
		\bigg(\int_{( \R^m)^i}\left(\Phi_{x_1}\ldots \Phi_{x_i}h_j(x,  y)\right)\left(\Phi_{x_1}\ldots \Phi_{x_i}h_j(x,  z)\right)d\nu ^{\otimes m  i}(x)\\
		&\ \ \ \ \times \int_{( \R^m)^i}\left(\Phi_{x'_1}\ldots \Phi_{x'_i}h_j(x',  y)\right)\left(\Phi_{x'_1}\ldots \Phi_{x'_i}h_j(x',  z)\right)
		d\nu ^{\otimes m  i}(x')\bigg)\\
		&\hspace{250pt}d\nu ^{\otimes m  (j-i)}(y)d\nu ^{\otimes m  (j-i)}(z)\\
		&=\int_{( \R^m)^{i}}\int_{( \R^m)^{i}}
		\bigg(\int_{( \R^m)^{j-i}}\left(\Phi_{x_1}\ldots \Phi_{x_i}h_j(x,  y)\right)\left(\Phi_{x'_1}\ldots \Phi_{x'_i}h_j(x',  y)\right)d\nu ^{\otimes m  j-i}(y)\bigg)^2\\
		&\hspace{270pt}d\nu ^{\otimes m  i}(x)d\nu ^{\otimes m  i}(x')\\
		&=\int_{( \R^m)^{i}}\int_{( \R^m)^{i}}
		\bigg(\Phi_{x_1}\ldots \Phi_{x_i}\Phi_{x'_1}\ldots \Phi_{x'_i}\int_{( \R^m)^{j-i}}h_j(x,  y)h_j(x',  y)d\nu ^{\otimes m  j-i}(y)\bigg)^2\\
		&\hspace{270pt}d\nu ^{\otimes m  i}(x)d\nu ^{\otimes m  i}(x')\\
		&{\le}\int_{( \R^m)^{i}}\int_{( \R^m)^{i}}
		\bigg(\int_{( \R^m)^{j-i}}h_j(x,  y)h_j(x',  y)d\nu ^{\otimes m  j-i}(y)\bigg)^2d\nu ^{\otimes m  i}(x)d\nu ^{\otimes m  i}(x')\\[2pt]
		&=\| h_j\ast_i^i h_j\|^2_{2}.
	\end{align*}
	
	Finally,   we pass to the proof of the last inequality. We start in the same manner as in the previous case.
	\begin{align}\label{aux3}
		&\|\db h_j\ast_i^i\db h_i\|^2_{2}\\\nonumber
		&=\int_{( \R^m)^{j-i}}
		\bigg(\int_{( \R^m)^i}\left(\Phi_{x_1}\ldots \Phi_{x_i}\Phi_{y_1}\ldots \Phi_{y_{j-i}}h_j(x,  y)\right)\\\nonumber
		&\hspace{150pt}\times\left(\Phi_{x_1}\ldots \Phi_{x_i}h_i(x)\right)d\nu ^{\otimes m  i}(x)\bigg)^2d\nu ^{\otimes m  (j-i)}(y)\\\nonumber
		&=\int_{( \R^m)^{j-i}}
		\bigg(\Phi_{y_1}\ldots \Phi_{y_{j-i}}\int_{( \R^m)^i}\left(\Phi_{x_1}\ldots \Phi_{x_i}h_j(x,  y)\right)\\\nonumber
		&\hspace{150pt}\times\left(\Phi_{x_1}\ldots \Phi_{x_i}h_i(x)\right)d\nu ^{\otimes m  i}(x)\bigg)^2d\nu ^{\otimes m  (j-i)}(y)\\ \nonumber
		&\le\int_{( \R^m)^{j-i}}
		\bigg(\int_{( \R^m)^i}\left(\Phi_{x_1}\ldots \Phi_{x_i}h_j(x,  y)\right)\left(\Phi_{x_1}\ldots \Phi_{x_i}h_i(x)\right)d\nu ^{\otimes m  i}(x)\bigg)^2d\\
		&\hspace{270pt}\nu ^{\otimes m  (j-i)}(y).\hspace{26pt}\nonumber
	\end{align}
	Due to the covariance identity
	\begin{align*}
		&\int_{ \R^m}\big(\Phi_{t}f(t)\big) \big(\Phi_{t}g(t)\big)d\nu^{\otimes m} (t)\\
		&= \int_{ \R^m}f(t)g(t)d\nu (t)-\int_{ \R^m}f(t)d\nu^{\otimes m} (t)\int_{ \R^m}g(t)d\nu^{\otimes m} (t),  
	\end{align*}
	where $f,  g\in L^2(\left( \R^m,   \nu^{\otimes m} \right)$,   and the formula \eqref{aux2},   we may rewrite the inner integral in \eqref{aux3} as follows
	\begin{align*}
		&\int_{( \R^m)^i}\left(\Phi_{x_1}\ldots \Phi_{x_i}h_j(x,  y)\right)\left(\Phi_{x_1}\ldots \Phi_{x_i}h_i(x)\right)d\nu ^{\otimes m  i}(x)\\
		&=\sum_{s=0}^i\sum_{1\le k_1<\ldots<k_s\le i}(-1)^{i-s}\int_{( \R^m)^{{s}}}h_{j-i+s}(x_{k_1},  \ldots,  x_{k_s},  y)h_{s}(x_{k_1},  \ldots,  x_{k_s})\\
		&\hspace{270pt} d\nu ^{\otimes m  i}(x_{k_1},  \ldots,  x_{k_s}).
	\end{align*}
	Applying this to \eqref{aux3},   we arrive at
	\begin{align*}\nonumber
		\|\db h_j\ast_i^i\db h_i\|^2_{2}&\lesssim \sum_{s=0}^i\int_{( \R^m)^{j-i}}
		\bigg(\int_{( \R^m)^s}h_{j-i+s}(x,  y)h_s(x)d\nu ^{\otimes m  s}(x)\bigg)^2d\nu ^{\otimes m  (j-i)}(y)\\
		&= \sum_{s=1}^i\|h_{j-i+s}\ast_s^s h_s\|^2_{2}+\|h_{j-i}\|^2_{2}\left(\int_{( \R^m)^{r}}h(x)d\nu ^{\otimes m  r}(x)\right)^2,  
	\end{align*}
	which ends the proof.
\end{proof}

We deduce the assertion of Theorem \ref{thm:prob->discrete} from Theorem \ref{thm:generaldb} and Lemma \ref{lem:db<} as follows.

\begin{proof}[Proof of Theorem~\ref{thm:prob->discrete}] Since $\db{\left(\overline h_j\right)}=\db h_j $,   Lemma \ref{lem:db<} gives us
	\begin{align*}
		\|\db h_j\ast_j^i\db h_j\|_2&\lesssim \| \overline h_j\ast_j^i  \overline h_j\|_2,  \\
		\|\db h_j\ast_i^i\db h_j\|_2&\lesssim \|  \overline h_j\ast_i^i  \overline h_k\|_2,  
	\end{align*}
	as well as
	\begin{align*}
		{\sum_{1\le i<j\le r}}	\frac{\|\db h_j\ast_i^i\db h_i\|^2_2}{n^{j+i}}&\lesssim {\sum_{1\le i<j\le r}}	\frac{1}{n^{j+i}} \sum_{s=1}^i\| \overline h_{j-i+s}\ast_s^s  \overline h_s\|^2_2\\
		&\le {\sum_{1\le i<j\le r}}	 \sum_{s=1}^i\frac{1}{n^{j-i+2s}}\| \overline h_{j-i+s}\ast_s^s  \overline h_s\|^2_2\\
		&= {\sum_{1\le i<j\le r}}\frac{r-j+i}{n^{j+i}}\|  \overline h_j\ast_i^i  \overline h_i\|^2_2\\
		&\leq r {\sum_{1\le i<j\le r}}\frac{\|  \overline h_j\ast_i^i  \overline h_i\|^2_2}{n^{j+i}},  
	\end{align*}
	where we used $\int_{( \R^m)^r}\overline h(x)d\nu ^{\otimes m  r}(x)=0$. Applying this to Theorem \ref{thm:generaldb},   we conclude the assertion of Theorem \ref{thm:prob->discrete}.
\end{proof}

\section{Theorem~\ref{thm:prob->discrete} and subgraph counts in $\Gnmp$}\label{sec:IntegralsToGraphs}

Our aim is to describe a method that can, in principle, be applied to an arbitrary subgraph. However, in the random intersection graph setting, even with this method at hand, the analysis involves considering a number of different subcases. We therefore focus on cliques, although the same line of reasoning can also be adapted to other subgraphs. As demonstrated in the companion paper \cite{Cliques}, such an extension requires a case-specific analysis. Consequently, in this section, we interpret in detail Theorem~\ref{thm:prob->discrete} in the setting where $H=K_r$. In the remainder of this paper we provide detailed calculations for the case of edges.

In order to apply Theorem~\ref{thm:prob->discrete} for subgraph counts,    we need to define  an appropriate function $h$ and a measure $\nu$ so that  $X$ represents  a subgraph count  in $\Gnmp$. Let $H$ be any graph on $r$ vertices. We denote by $\Subgraphs_H$ the family of all isomorphic copies of  $H$ in the complete graph with the vertex set $\V$. We also define a shorthand notation, for any ${\bf H}\in \Subgraphs_H$
\begin{equation}\label{Eq:Def:Pi}
	\pi(H):=\Pra{{\bf H}\subseteq \Gnmp}.
\end{equation}


With each vertex $v$ we  associate a vector ${\bf t}_v=(t_{v,  1},    \ldots,    t_{v,  m})\in \{0,   1\}^m$ that represents a sample of $A^{(v)}$ (which is defined at the beginning of Section~\ref{sec:mainresults}). Furthermore,  for any vertex set $\{v_1, \ldots, v_r\}\subseteq \V$ we associate a vector  $\ttt=\ttt_{H}=({\bf t}_{v_1},    \ldots,    {\bf t}_{v_r})$. We define
\begin{equation*}
	g(\ttt)=\sum_{{{\bf H}\in\Subgraphs_{H}, V({\bf H})=\{v_1, \ldots, v_r\}}}\ind_{\{\forall_{vv'\in E({\bf H})}\exists_{j} t_{v,  j}=t_{v',  j}=1\}},  \\
\end{equation*}
where $\ind_{\mathcal{A}}$ is an indicator of the event $\mathcal{A}$. 

As pursuing the subsequent analysis in full generality would result in complicated and cumbersome formulas, we restrict our attention to a specific class of subgraphs in order to make the main steps of the argument more transparent. More precisely, we focus on $H=K_r$, the clique on $r$ vertices. In this case
\begin{equation}\label{Eq:g}
	g(\ttt)=
	\begin{cases}
		1,   &\text{ if }\forall_{v,v'\in \{v_1, \ldots, v_r\}}\exists_{j} t_{v,  j}=t_{v',  j}=1,  \\
		0,   &\text{ else.}
	\end{cases}
\end{equation}

This definition  implies that,  for any set of $r$ vertices $\{v_1, \ldots, v_r\}\subseteq \V$,  the equality 
\begin{equation}\label{Eq:Intro:g}
	g(A^{(v_1)},   A^{(v_2)},   \ldots,   A^{(v_r)})=1
\end{equation} 
holds if and only if the set $\{v_{1},   v_{2},   \ldots,   v_{r}\}$ induces $K_r$ in the random intersection graph $\Gnmp$. Thus, for $\X$ -- the number of subgraphs isomorphic to $K_r$,     we may write 
\begin{align*}
	\X=\sum_{\{v_{1},   v_{2},   \ldots,   v_{r}\}\subseteq \V}g(A^{(v_1)},   A^{(v_2)},   \ldots,   A^{(v_r)}), 
\end{align*}
which is the form of $\X$ that allows us to use Theorem \ref{thm:prob->discrete}. Recall that we have associated $\ttt=({\bf t}_{v_1},    \ldots,    {\bf t}_{v_r})$ with the clique $\kli\in\Subgraphs_{K_r}$ with the vertex $\{v_1, \ldots, v_k\}$. Since  $\mu_{m,   p}$ is the distribution of each of $A^{(v)}$,    we have
\begin{equation}\label{Eq:IntroG2}
	\int_{\left(\{0,   1\}^m\right)^{r}}g(\ttt)d\mu_{m,   p}^{\otimes(r)}(\ttt)=\Pra{\kli\subseteq \Gnmp}=\pi(K_r)=:\qq,
\end{equation}
where $\pi(\cdot)$ is defined in \eqref{Eq:Def:Pi}. Furthermore,  for $0\le j\le r$ and $\ttt_j=({\bf t}_{v_1}, \ldots, {\bf t}_{v_j})\in (\{0,     1\}^m)^j$,     from \eqref{eq:bar} and \eqref{eq:defh_j} we have 
\begin{align*}
	g_r(\ttt)
	&=g(\ttt),   &&
	\\
	g_j(\ttt_j)
	&=\int_{(\{0,     1\}^m)^{r-j}}g(\ttt_j,   {\bf t}_{v_{j+1}},  \ldots {\bf t}_{v_{r}})d\mu_{m,     p}^{\otimes(r-j)} ({\bf t}_{v_{j+1}},  \ldots {\bf t}_{v_{r}}),   & 1\le j&\le r-1,   \\
	\overline g_j(\ttt_j)&=g_j(\ttt_{j})-\qq,    & 1\le j&\le r.
\end{align*}
Then,       Theorem~\ref{thm:prob->discrete} ensures existence of a constant $C'=C'(r)>0$ such that
\begin{multline}
	d_{K/W}\left( \widetilde \X ,      \mathcal{N}\right)
	\le\frac{C'}{\Var [\X]} 
	\Bigg(
	{\sum_{0\le i<j\le r}}n^{4r-3j+i}
	\| \overline g_j\ast_j^i  \overline g_j\|^2_2
	\\+
	{\sum_{1\le i<j\le r}}
	\bigg(n^{4r-2j}\| \overline g_j\ast_i^i  \overline g_j\|^2_2
	+n^{4r-j-i}\| \overline g_j\ast_i^i  \overline g_i\|^2_2
	\bigg)\Bigg)^{1/2}.\label{Eq:dkW}
\end{multline}
Here,    we follow the notation from \eqref{eq:Lpnotation} with $\nu^{\otimes m}=\mu_{m,    p}$,    i.e., for $f\in L^2\big((\{0,   1\}^m)^k,     \mu_{m,    p}^{\otimes k}\big)$,     $k\in\N$,    we write
$$\|f\|_2:=\|f\|_{L^2\left((\{0,   1\}^m)^k,     \mu_{m,    p}^{\otimes k}\right)}.$$

%
%

\begin{figure}
	\begin{center}
		\includegraphics[width=10cm]{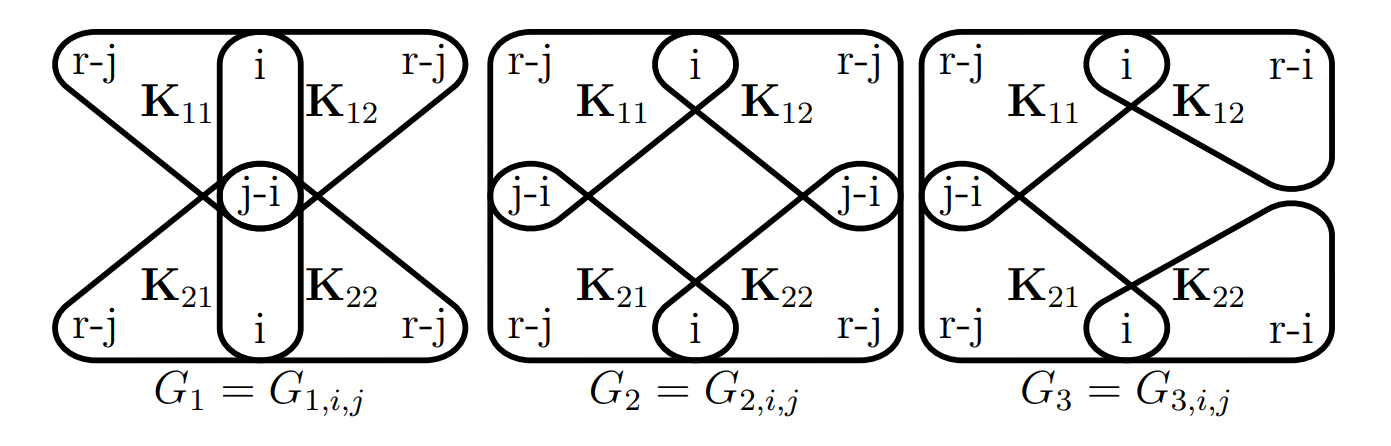}
		\caption{\label{Fig:GraphRepresentations} Symbolic representations of graphs corresponding to integrals 
			$\| g_j\ast_j^i g_j\|^2$,   
			$\| g_j\ast_i^i g_j\|^2$,    and
			$\| g_j\ast_i^i g_i\|^2$,    resp..
		}
	\end{center}
\end{figure}

Our next aim is to represent probabilistically the integrals occurring in \eqref{Eq:dkW}.
Namely,  we will express them by means of probabilities of the graphs $G_{t,   i,   j}$,   $t\in\{1,   2,   3\}$,   $0\le i < j \le r$,  defined as follows.

Let $\kli_{ab}$,  $ab\in \indeksy=\{11,  12,  21,  22\}$,  be four cliques  of size $r$. We put
\begin{equation}\label{Eq:DefGt}
	G_{t,  i,  j}=\bigcup_{ab\in \indeksy}\kli_{ab}, 
\end{equation}
where the  cliques $\kli_{ab}$,  $ab\in \indeksy$,   intersect in a special manner depending on $t,  i,  j$,  which is  described below.
In  Figure~\ref{Fig:GraphRepresentations} one may find the geometrical idea behind the definition,   and in Figure~\ref{Fig:GtExamples} we present some particular cases.

For $t=1$ and $0\le i < j \le r$
we define the vertex set of $G_{1,   i,   j}$ by 
\begin{align*}
	V(G_{1,  i,  j})=&
	\{v_{l,  a}: l\in \{1,  \ldots,  i\},   a\in \{1,  2\}\}\\
	&\cup
	\{w_{l}: l\in \{1,  \ldots,  j-i\}\}\\
	&\cup 
	\{u_{l,  ab}: l\in \{1,  \ldots,  r-j\},   a,  b\in \{1,  2\}\}.
\end{align*}
Moreover,   for each $ab\in \indeksy$ we define $\kli_{ab}$ in \eqref{Eq:DefGt} to be a clique on $r$ vertices with the vertex set 
$$\{v_{l,  a}: l\in \{1,  \ldots,  i\}\}\cup  \{w_{l}: l\in \{1,  \ldots,  j-i\}\}\cup \{u_{l,  ab}: l\in \{1,  \ldots,  r-j\}\}.$$
One can observe that the vertices $\{w_l\}$ are shared by all of the cliques $\kli_{ab}$,  $ab\in \indeksy$. In addition,  given $a\in \{1, 2\}$,   the vertices $\{v_{l,   a}\}$ are shared by two cliques $\kli_{a1}$ and $\kli_{a2}$.   Moreover,  given $ab\in \indeksy$,   the vertices $\{u_{l,   ab}\}$ belong to $\kli_{ab}$ only.

Note that for $t=1$ and $j=r$ we have $\kli_{11}=\kli_{12}$ and $\kli_{21}=\kli_{22}$, which implies the next fact. While this case is special, the same proof methods will nevertheless apply.

\begin{fact}\label{Fact:G1r}
	$G_{1, i, r}$ is isomorphic to a sum of two cliques $K_r$ intersecting on $r-i$ vertices. In particular $G_{1, 0, r}\cong K_r$. 
\end{fact}

In a similar manner we define 
$G_{2,  i,  j}$,   $1\le i < j \le r$.  Its vertex set is
\begin{align*}
	V(G_{2,  i,  j})=&
	\{v_{l,  a}: l\in \{1,  \ldots,  i\},   a\in \{1,  2\}\}\\
	&\cup
	\{w_{l,  b}: l\in \{1,  \ldots,  j-i\},   b\in \{1,  2\}\}\\
	&\cup 
	\{u_{l,  ab}: l\in \{1,  \ldots,  r-j\},   a,  b\in \{1,  2\}\}.
\end{align*}
Moreover,   for each $ab\in \indeksy$,   we set $\kli_{ab}$ in \eqref{Eq:DefGt} to be a clique on $r$ vertices with the vertex set 
$$\{v_{l,  a}: l\in \{1,  \ldots,  i\}\}\cup  \{w_{l,  b}: l\in \{1,  \ldots,  j-i\}\}\cup \{u_{l,  ab}: l\in \{1,  \ldots,  r-j\}\}.$$

\begin{figure}\begin{center}
		\includegraphics[width=15cm]{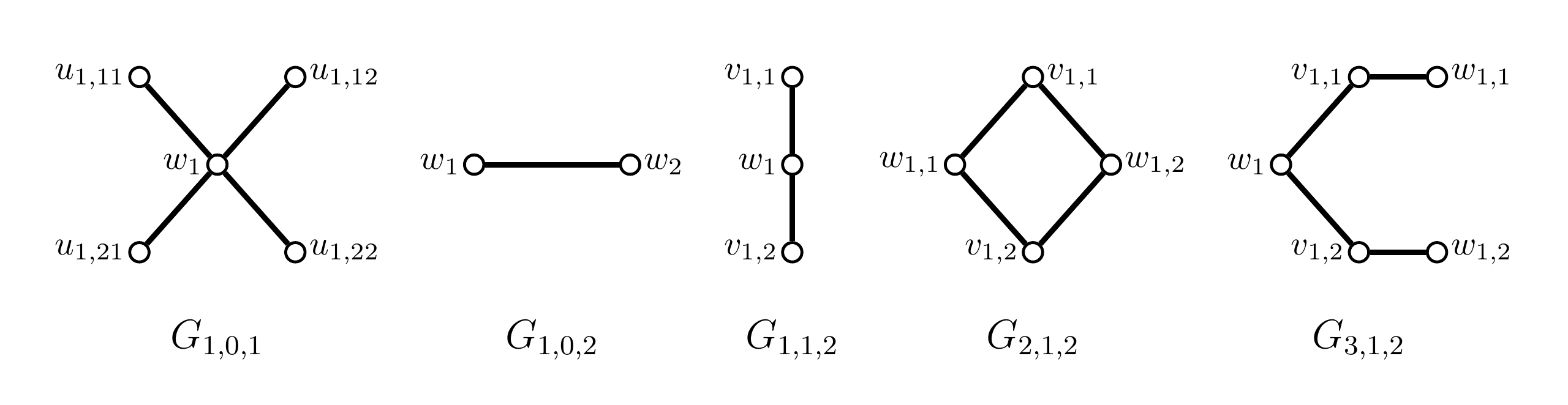}
	\end{center}
	\caption{\label{Fig:GtExamples}  Examples of $G_{t,  i,  j}$ for $r=2$.}
\end{figure}

Eventually,   for $G_{3,  i,  j}$,   $1\le i < j \le r$,   we set 
\begin{align*}
	V(G_{3,  i,  j})=&
	\{v_{l,  a}: l\in \{1,  \ldots,  i\},   a\in \{1,  2\}\}\\
	&\cup
	\{w_{l}: l\in \{1,  \ldots,  j-i\}\}\\
	&\cup
	\{w_{l,  a}: l\in \{1,  \ldots,  j-i\},   a\in \{1,  2\}\}\\
	&\cup 
	\{u_{l,  ab}: l\in \{1,  \ldots,  r-j\},   a,  b\in \{1,  2\}\}.
\end{align*}
For each $a\in \{1,  2\}$ we define $\kli_{a1}$ to be a clique on $r$ vertices with the vertex set 
$$\{v_{l,  a}: l\in \{1,  \ldots,  i\}\}\cup  \{w_{l}: l\in \{1,  \ldots,  j-i\}\}\cup \{u_{l,  a1}: l\in \{1,  \ldots,  r-j\}\}.$$
and 
$\kli_{a2}$ to be a clique on $r$ vertices with the vertex set 
$$\{v_{l,  a}: l\in \{1,  \ldots,  i\}\}\cup  \{w_{l,  a}: l\in \{1,  \ldots,  j-i\}\}\cup \{u_{l,  a2}: l\in \{1,  \ldots,  r-j\}\}.$$

In the sequel,   we will omit indices $i,  j$ and write $G_t=G_{t,  i,  j}$,   when it will be clear from the context which values of $i,   j$ are considered.

Given  $i,   j$,   for $t\in \{1,    2,    3\}$,     $\II\subseteq \indeksy$ and  $G_t=G_{t,    i,    j}$,    we define 
\begin{align}\label{eq:GtI}
	G_{t,   \II}:=\bigcup_{ab\in\II}\kli_{ab},  
\end{align}
i.e., $G_{t,   \II}$ is the subgraph of $G_t$ induced by the vertex sets of the cliques with indices in~$\II$. In addition,  we use the convention that $G_{t,   \emptyset}$ has neither vertices nor edges and $\pi(G_{t,   \emptyset})=1$. 

The following lemma relates the integrals from \eqref{Eq:dkW} 
to probabilities of existence of graphs $G_{t,  \II}$.

\begin{lemma}\label{Lem:CalkiDoPodgrafy}
	For
	\begin{align*}
		t&=1	&i_0&=i,    &i_1=i_2=i_3&=j,    &0\le i &< j\le r;\\
		t&=2	&i_0=i_2&=i,    &i_1=i_3&=j,    &1\le i &< j\le r;\\
		t&=3	&i_0=i_2=i_3&=i,    &i_1&=j,    &1\le i &< j\le r,   \\
	\end{align*}
	we have
	$$
	\left\| \bar g_{i_1}\ast^{i_0}_{i_2} \bar g_{i_3}\right\|_2^2=\sum_{\II\subseteq\indeksy}(-1)^{|\II|}\qq^{|\indeksy|-|\II|}\pi(G_{t,   \II}),   
	$$

\end{lemma}

Let us also note that for any $t=1,  2,  3$ we have
\begin{align}\label{eq:|V|}
	|V (G_t)| = 4r - i_1 - i_2 - i_3 + i_0,  
\end{align}
which may be verified by a simple inspection of graphs. Thus,  the values of $|V(G_t)|$ coincide with the powers of $n$ in \eqref{Eq:dkW},  so we can interpret the terms appearing there as the expected number of copies of $G_t$,  $t=1, 2, 3$ for $\II=\indeksy$. This, together with \eqref{Eq:dkW} imply.

\begin{corollary}\label{Cor:Main}
	Let $t=1,2,3$ and $G_t=G_{t,i,j}$ be defined as above. Then there exists a constant $C'=C'(r)>0$ such that
	\begin{equation*}
		d_{K/W}\left( \widetilde \X ,      \mathcal{N}\right)
		\le\frac{C'}{\Var [\X]} 
		\sum_{t,i,j}\left(n^{|V (G_t)|} \sum_{\II\subseteq\indeksy}(-1)^{|\II|}\qq^{|\indeksy|-|\II|}\pi(G_{t,   \II})\right).
	\end{equation*}
\end{corollary}

\begin{proof}[Proof of Lemma~\ref{Lem:CalkiDoPodgrafy}]
	First,  let us consider the first type of the graphs,   i.e. $t=1$. 	
	Recall that
	\begin{align*}
		V(G_{1,  i,  j})=&
		\{v_{l,  a}: l\in \{1,  \ldots,  i\},   a\in \{1,  2\}\}\\
		&\cup
		\{w_{l}: l\in \{1,  \ldots,  j-i\}\}\\
		&\cup 
		\{u_{l,  ab}: l\in \{1,  \ldots,  r-j\},   a,  b\in \{1,  2\}\}.
	\end{align*}	
	Now,   to each vertex of $G_{1,  i,  j}$ we assign a distinct vertex from $\V$ in $\Gnmp$. Moreover,  as explained above,  with each vertex $v_{l,  a}$ we associate a vector ${\bf x}_{l,  a}=(x_{1,  l,  a},  \ldots,  x_{m,  l,  a})\in \{0,  1\}^m$.
	Also,   to each vertex $w_{l}$ and $u_{l,  ab}$ we associate  a vector ${\bf y}_{l}\in \{0,  1\}^m$,   and ${\bf z}_{l,  ab}\in \{0,  1\}^m$,   respectively.
	Note that,   by the definition \eqref{Eq:g} of the function $g$,   for $ab\in \indeksy$,  using the same reasoning that led to \eqref{Eq:IntroG2},  we note that 
	$$
	g({\bf x}_{1,  a},  \ldots,  {\bf x}_{i,  a},  {\bf y}_{1},  \ldots,  {\bf y}_{j-i},  {\bf z}_{1,  ab},  \ldots,  {\bf z}_{r-j,  ab})=1
	$$
	if and only if  $\Gnmp$, conditioned to satisfy  
	\begin{alignat*}{3}A^{(v_{1, a})}&={\bf x}_{1,  a},  \ &&\ldots,  \ A^{(v_{i, a})}&&={\bf x}_{i,  a},\\  A^{(w_{1})}&={\bf y}_{1},  &&\ldots, \  A^{(w_{j-i})}&&={\bf y}_{j-i},\\
		A^{(u_{1, ab})}&={\bf z}_{1,  ab},  &&\ldots, \   A^{(u_{r-j, ab})}&&={\bf z}_{r-j,  ab},
	\end{alignat*}
	contains the clique  on the vertices assigned to $V(\kli_{ab})$. 
	Let us introduce a shorthand notation 
	$$
	\xx_a=({\bf x}_{1,  a},  \ldots,  {\bf x}_{i,  a}),  \ \yy=({\bf y}_{1},  \ldots,  {\bf y}_{j-i}),  \ \zz_{ab}=({\bf z}_{1,  ab},  \ldots,  {\bf z}_{r-j,  ab}).  
	$$
	Then,   as a consequence of the previous observation,   for any $\II\subseteq \indeksy$	
	\begin{align}\label{eq:g->G}
		\prod_{ab\in\II}
		g(\xx_a,   \yy,   \zz_{ab})=1,   
	\end{align} 
	holds if and only if in $\Gnmp$,  there is a copy  of $G_{1,  i,  j}$ on the vertices from $\V$ assigned to  $V(G_{1,  i,  j})$. 		
	
	Now we may proceed to analyse the value of the integral for $t=1$. For convenience we introduce a shorthand notation $\int f(\ttt) d\ttt:= \int_{(\{0,     1\}^m)^k} f(\ttt) d\mu^{\otimes k}_{m,     p}(\ttt)	$,   for $\ttt\in (\{0,    1\})^k$. Then 
	\begin{align*}
		\left\| \bar g_{j}\ast^i_j \bar g_{j}\right\|_2^2&=
		\int \left(\int\bar g_j^2(\xx_1,   \yy)d\xx_1\right)^2d\yy \\
		&=\int \int\int\bar g_j^2(\xx_1,   \yy)\bar g_j^2(\xx_2,   \yy)d\xx_1 d\xx_2 d\yy \\
		&=\int\ldots\int(g(\xx_1,    \yy,     \zz_{11})-\qq)(g(\xx_1,   \yy,     \zz_{12})-\qq)\times\\
		&\hspace{19mm}(g(\xx_2,   \yy,     \zz_{21})-\qq)(g(\xx_2,   \yy,   \zz_{22})-\qq)d\xx_1 d\xx_2 d\yy d\zz_{11}d\zz_{12}d\zz_{21}d\zz_{22}\\
		&=\int\ldots\int\sum_{\II\subseteq\indeksy}(-1)^{|\II|}\left(\prod_{ab\in\II}g(\xx_a,   \yy,   \zz_{ab})\right)\qq^{|\indeksy|-|\II|}d\xx_1 d\xx_2 d\yy d\zz_{11} d\zz_{12}d\zz_{21}d\zz_{22}\\
		&=\sum_{\II\subseteq\indeksy}(-1)^{|\II|}\qq^{|\indeksy|-|\II|}\int\ldots\int\left(\prod_{ab\in\II}g(\xx_a,   \yy,   \zz_{ab})\right)d\xx_1 d\xx_2 d\yy d\zz_{11}d\zz_{12}d\zz_{21}d\zz_{22}\\
		&=\sum_{\II\subseteq\indeksy}(-1)^{|\II|}\qq^{|\indeksy|-|\II|}\pi(G_{1,   \II}).
	\end{align*}
	Recall that the integrals are taken over a probability space; therefore, integration with respect to a variable that does not appear in the integrand has no effect on the value of the integral.
	
	We use an analogous reasoning for the remaining two integrals. 
	To each vertex of $G_{2,  i,  j}$ we associate a vertex in $\Gnmp$. Moreover,  to the vertices from 	$\{v_{l,  a}: l\in \{1,  \ldots,  i\},   a\in \{1,  2\}\}$,   
	$
	\{w_{l,  b}: l\in \{1,  \ldots,  j-i\},   b\in \{1,  2\}\}
	$,   and $ 
	\{u_{l,  ab}: l\in \{1,  \ldots,  r-j\},   a,  b\in \{1,  2\}\}$ we assign vectors $\xx_{a}$,   $\yy_b$,   and $\zz_{ab}$,   $ab\in\indeksy$,   respectively.
	Therefore,   similarly as in the previous case,   we obtain
	\begin{align*}
		\left\| \bar g_{j}\ast^i_i \bar g_{j}\right\|_2^2&=
		\int \left(\int\bar g_j(\xx_1,   \yy_1)\bar g_j(\xx_1,   \yy_2)d\xx_1\right)^2d\yy_1d\yy_2 \\
		&=\int\ldots\int(g(\xx_1,    \yy_1,     \zz_{11})-\qq)(g(\xx_1,   \yy_2,     \zz_{12})-\qq)\times\\
		&\hspace{19mm}(g(\xx_2,   \yy_1,     \zz_{21})-\qq)(g(\xx_2,   \yy_2,   \zz_{22})-\qq)d\xx_1 d\xx_2 d\yy_1 d\yy_2 d\zz_{11}d\zz_{12}d\zz_{21}d\zz_{22}\\
		&=\int\ldots\int\sum_{\II\subseteq\indeksy}(-1)^{|\II|}\left(\prod_{ab\in\II}g(\xx_a,   \yy_b,   \zz_{ab})\right)\qq^{|\indeksy|-|\II|}d\xx_1 d\xx_2 d\yy_1 d\yy_2 d\zz_{11}d\zz_{12}d\zz_{21}d\zz_{22}\\
		&=\sum_{\II\subseteq\indeksy}(-1)^{|\II|}\qq^{|\indeksy|-|\II|}\int\ldots\int\left(\prod_{ab\in\II}g(\xx_a,   \yy_b,   \zz_{ab})\right)d\xx_1 d\xx_2 d\yy_1 d\yy_2 d\zz_{11}d\zz_{12}d\zz_{21}d\zz_{22}\\
		&=\sum_{\II\subseteq\indeksy}(-1)^{|\II|}\qq^{|\indeksy|-|\II|}\pi(G_{2,   \II}).
	\end{align*}
	Last but not least,   in the case of the third integral,   after defining $\xx_{a}$,   $\yy$,   $\yy_a$,   and $\zz_{ab}$,   $ab\in\indeksy$ in an analogous manner,   we have 
	\begin{align*}
		\left\| \bar g_{j}\ast^i_i \bar g_{i}\right\|_2^2&=
		\int \left(\int\bar g_i(\xx)\bar g_j(\xx,   \yy)d\xx\right)^2d\yy \\
		&=\int\ldots\int(g(\xx_1,    \yy_1,     \zz_{11})-\qq)(g(\xx_1,   \yy,     \zz_{12})-\qq)\times\\
		&\hspace{19mm}(g(\xx_2,   \yy_2,     \zz_{21})-\qq)(g(\xx_2,   \yy,   \zz_{22})-\qq)d\xx_1 d\xx_2 d\yy d\yy_1 d\yy_2d\zz_{11}d\zz_{12}d\zz_{21}d\zz_{22}\\
		&=\sum_{\II\subseteq\indeksy}(-1)^{|\II|}\qq^{|\indeksy|-|\II|}\pi(G_{3,   \II}).
	\end{align*}
	This concludes the proof.
\end{proof}

\section{Edge count - preliminaries}
\label{sec:preliminaries}

The expected value of the number of edges $\Xe$ is clearly given by $\E[\Xe]={n\choose 2}\hatp$. However, the variance of $\Xe$ cannot be expressed solely in terms  of $n$ and $\hatp$. In addition to the term ${n\choose 2}\hatp(1-\hatp)$,  which represents the variance of the number of edges in the $\Gnp$ model,  another term is required. Below,  we derive  
the exact formula and two-sided estimates for the variance of $\Xe$.

\begin{lemma}\label{Lem:Variance}
	Let $\Xe$ be the number of edges in $\Gnmp$,  and $n,  m\geq3$. We have
	\begin{align}\label{eq:vareq}
		\Var [\Xe] &={n\choose 2}\hatp(1-\hatp)+{6}{n\choose 3}\[\big(1-2p^2+p^3\big)^m-(1-\hatp)^{2}\]\\[5pt]
		\label{eq:varest}
		&\asymp  n^2\hatp(1-\hatp)+n^3 (1-2p^2+p^3)^m\[1\wedge (mp^3)\].
	\end{align}
	In particular,  if $mp^3\le 1$,  then 
	\begin{equation}\label{eq:var-mp3small}
		\Var [\Xe] {\, \asymp\,  } n^2\hatp(1-\hatp)+n^3mp^3(1-\hatp)^2.
	\end{equation}
\end{lemma}

\begin{proof}
	Since edges connecting separate pairs of vertices are independent,  we get

	\begin{align}\nonumber
		\Var [\Xe] &= \sum_{e}\Var [\ind_{e}] + \sum_{e,  e': |e\cap e'|=1}\cov (\ind_{e},  \ind_{e'})\\\label{Eq:Var01}
		&= {n\choose 2}\Var [\ind_{e_1}] + {6}{n\choose 3}\cov (\ind_{e_1},  \ind_{e_2}),  
	\end{align}
	where $\ind_{e}$  stands for the indicator of the existence of an edge $e$,  and $e_1$ and $e_2$ are two edges with one common vertex.
	We clearly have
	\begin{align}\label{eq:VarI_e}
		\Var [\ind_{e_1}]=\E[\ind_{e_1}]-(\E[\ind_{e_1}])^2=\hatp -\hatp^2, 
	\end{align}
	while the covariance of $\ind_{e_1}$ and $\ind_{e_2}$ takes the form
	\begin{align*}
		\cov (\ind_{e_1},  \ind_{e_2})
		&=\cov ((1-\ind_{e_1}), (1-\ind_{e_2}))\\
		&=\E [(1-\ind_{e_1})(1-\ind_{e_2})] - \E (1-\ind_{e_1})\E (1-\ind_{e_2})\\
		&=\P\left(e_1,  e_2\not\in \Gnmp\right)- \left(1-\hatp\right)^2.
	\end{align*}
	The simultaneous non-existence of the edges $e_1$ and $e_2$ occurs whenever any attribute $a$ is not chosen by both vertices of $e_1$ or by both vertices of $e_2$. Thus,  each attribute is either not chosen by the common vertex $v$ of $e_1$ and $e_2$,  or it is chosen by $v$ but not by the other two endpoints of $e_1,  e_2$. Thus,  we conclude
	\begin{align}\label{eq:ee}
		\P\left(e_1,  e_2\not\in \Gnmp\right)=\big(1-p+p(1-p)^2\big)^m=\big(1-2p^2+p^3\big)^m, 
	\end{align}
	and consequently 
	$$\cov (\ind_{e_1},  \ind_{e_2})=\big(1-2p^2+p^3\big)^m-(1-\hatp)^{2}, $$
	which proves \eqref{eq:vareq}. In order to estimate the difference above,  we focus on  the ratio
	\begin{align}\label{eq:ratio1}
		\frac{(1-\hatp)^{2}}{\big(1-2p^2+p^3\big)^m}=
		{\left(\frac{1-2p^2+p^4}{1-2p^2+p^3}\right)^m=}
		\left(1-\frac{p^3}{1+p(1-p)}\right)^m.
	\end{align}
	Using the inequalities $e^{-u/(1-u)}\le 1-u\le e^{-u}$,  {$u\in (0,  1)$},  we may bound
	\begin{align}\label{eq:ratio2}
		e^{-mp^3/(1-p)}\le \left(1-\frac{p^3}{1+p(1-p)}\right)^m\le e^{-mp^3/2}.
	\end{align}
	Since $1-e^{-mp^3/(1-p)}\asymp  1-e^{-mp^3/2}\asymp  1\wedge (mp^3)$,  we arrive at
	\begin{align*}
		\cov (\ind_{e_1},  \ind_{e_2})&=(1-2p^2+p^3)^m\[1-\left(1-\frac{p^3}{1+p(1-p)}\right)^m\]\\
		&\asymp  (1-2p^2+p^3)^m\[1\wedge (mp^3)\].
	\end{align*}
	Applying this and \eqref{eq:VarI_e} to \eqref{Eq:Var01} we obtain the first part of the assertion. 
	The other one follows then by \eqref{eq:ratio1} and \eqref{eq:ratio2},  since for $mp^3\le 1$ one has $p\leq m^{-1/3}\leq 3^{-1/3}<1$.
\end{proof}

Now we prove a general lemma concerning intergrals in Theorem~\ref{thm:prob->discrete} for $r=2$ that apply to the whole considered range of parameters.
Before that we  note that for $r=2$, in the definition of $G_t=G_{t,i,j}$ introduced in Section~\ref{sec:IntegralsToGraphs} (see Figure~\ref{Fig:GraphRepresentations} and Figure~\ref{Fig:GtExamples}), each clique $\kli_{ab}$ can be identified with an edge $e_{ab}$ of the graph $G_t=G_{t,i,j}$. 
Then, for $t\in \{1,2,3\}$ and $\II\subseteq \indeksy:=\{11,  12,  21,  22\}$, we have $G_{t,\II}=G_t[\{e_{ab}:ab\in \II\}]$ the subgraph of $G_t$ induced by the edges with indices in~$\II$. Additionally, by convention $G_{t,\emptyset}$ has neither vertices nor edges and $\pi(G_{t,\emptyset})=1$.

\begin{lemma}\label{lem:2norms} Let $r=2$. For $n,  m\geq3$ we have for $G_1=G_{1,i,2}$, $i=0,1$
	\begin{equation*}
		n^{1+i}\left\|\overline g_{2}\ast_{2}^i \overline g_{2}\right\|^2_2
		=n^{|V(G_1)|}\sum_{\II\subseteq\indeksy}(-1)^{|\II|}\pp^{|\indeksy|-|\II|}\pi(G_{1,   \II})
		\lesssim \left(\Var [N_E]\right)^2\left(\frac1{n^2\hatp(1-\hatp)}+\frac1n\right).
	\end{equation*}
\end{lemma}

\begin{proof}
	Let $G_1=G_{1,0,2}$. By Fact~\ref{Fact:G1r} for all $\II\neq \emptyset$, $G_1\cong K_2$. Thus, in view of \eqref{eq:varest} we get
	\begin{align*}
		&n^{|V(G_1)|}\left|\sum_{\II\subseteq \indeksy}(-1)^{|\II|}\pp^{|\indeksy|-|\II|}\pi(G_{1,   \II})\right|
		=n^2\left|\sum_{\emptyset\neq\II\subseteq \indeksy}(-1)^{|\II|}\pp^{4-|\II|+1}+\pp^4\right|\\
		&=n^2\left|\pp\sum_{\II\subseteq \indeksy}(-1)^{|\II|}\pp^{4-|\II|}+\pp^4-\pp^5\right|
		=n^2\pp(1-\pp)^4+n^2\pp^4(1-\pp)\\
		&\asymp n^2\pp(1-\pp) 
		\lesssim \frac{(\Var[\Xe])^2}{n^2\pp(1-\pp)}. 
	\end{align*} 
	Consider now $G_1=G_{1,1,2}$. 
	By Fact~\ref{Fact:G1r} for $\II\neq \emptyset$, we have either $G_{1,\II}\cong K_2$ or $G_{1,\II}\cong G_{1,0,2}$, where $G_{1,0,2}$ is the union of two edges sharing a common vertex. Therefore, a careful examination of all possible cases yields
	\begin{align*}
		&n^{|V(G_1)|}\left|\sum_{\II\subseteq \indeksy}(-1)^{|\II|}\pp^{|\indeksy|-|\II|}\pi(G_{1,   \II})\right|\\
		&=n^3\left[(1-4\pp+4\pp^2)\pi(G_{1,0,2})+(2\pp^2-4\pp^3)\pi(\kli)+\pp^4\right]\\
		&=n^3\left[(1-4\pp+4\pp^2)\pi(G_{1,0,2})+(-\pp^2+4\pp^3-4\pp^4) +(\pp^2-2\pp^3+\pp^4) \right]\\
		&=n^3\left[(1-4\pp+4\pp^2)\left(\pi(G_{1,0,2})-\pp^2)+\pp^2(1-\pp)^2\right)\right]
	\end{align*}
	Note that  $(1-4\pp+4\pp^2)=(1-2\pp)^2\le 1$. Moreover, since $G_{1,0,2}$ is a union of two intersecting edges $\pi(G_{1,0,2})-\pp^2$ is the covariance appearing in \eqref{Eq:Var01}. Therefore in view of \eqref{eq:varest}  and  \eqref{Eq:Var01}
	\begin{align*}
		n^{|V(G_t)|}\left|\sum_{\II\subseteq \indeksy}(-1)^{|\II|}\pp^{|\indeksy|-|\II|}\pi(G_{t,   \II})\right|
		&\lesssim \Var[\Xe] + \frac{1}{n} (\Var[\Xe])^2\\
		&\lesssim (\Var[\Xe])^2\left(\frac{1}{n^2\pp(1-\pp)}+\frac1n\right).
	\end{align*}
	Finally, the assertion of the lemma follows from the above calculations together with Corollary~\ref{Cor:Main} and $\qq=\pp$ for $r=2$.
\end{proof}

\section{The case $mp^3\le 1$}\label{sec:CCs}

In this section we prove the following lemma.

\begin{lemma}\label{Lem:GIprawd} 
	Let $mp^3\le 1$ and $m\to\infty$ it holds. Then for $$G_t\in \{G_{1,0,2},G_{1,1,2},G_{1,0,1},G_{2,1,2},G_{3,1,2}\}$$ 
	we have
	\begin{multline}\label{eq:best}
		n^{|V(G_t)|}\left|\sum_{\II\subseteq \indeksy}(-1)^{|\II|}\pp^{|\indeksy|-|\II|}\pi(G_{t,   \II}),  \covers_{\indeksy}(H_{t,\II}))\right|\\
		\lesssim
		(\Var[\Xe])^2\left(\frac{1}{n^2\pp(1-\pp)}+\frac1n+\left(\frac{1}{nmp}\wedge \frac1m \frac{mp^2(1-\hatp)}{\hatp}\right)\right)	
	\end{multline}	
\end{lemma}
Lemma~\ref{Lem:GIprawd} and Corollary~\ref{Cor:Main} together establish part of the statement of Theorem~\ref{thm:Main}, as summarised in the following lemma.
\begin{lemma}\label{lem:mp^3<1}
	For $mp^3\le 1$  and $n,m$ large enough we have
	\begin{align*}
		d_{K/W}\left(\widetilde N_E,   \mathcal N\right)&\lesssim	\left(\dfrac1{n^2\hatp(1-\hatp)}\right)^{1/4}, 
	\end{align*}
	as well as
	\begin{align*}
		d_{K/W}\left(\widetilde N_E,   \mathcal N\right)
		&\lesssim	\left(\dfrac1{n^2\hatp(1-\hatp)}+\dfrac1n+\dfrac1{m}\right)^{1/2}.
	\end{align*}
\end{lemma}

\begin{proof} In Corollary~\ref{Cor:Main}, we have $\qq=\pp$ for $r=2$. Hence, it remains to prove that the right--hand side of \eqref{eq:best} can be bounded from above by the right-hand sides of the equations in Lemma~\ref{lem:mp^3<1}.
	\begin{align*}
		\frac{1}{nmp}=\frac{1}{\sqrt m \sqrt{n^2mp^2}}\lesssim \frac{1}{\sqrt m \sqrt{n^2\hatp}}\le \frac{1}{ \sqrt{n^2\hatp(1-\hatp)}}, 
	\end{align*}
	and consequently,  in view of Corollary~\ref{Cor:Main} and Lemma~\ref{Lem:GIprawd},  we obtain
	\begin{align*}
		d_{K/W}\left(\widetilde \Xe,   \mathcal N\right)
		&\lesssim	\left(\dfrac1{n^2\hatp(1-\hatp)}+\frac{1}{ \sqrt{n^2\hatp(1-\hatp)}}\right)^{1/2}\wedge 1\lesssim \left(\dfrac1{n^2\hatp(1-\hatp)}\right)^{1/4}. 
	\end{align*}
	
	On the other hand,  applying the inequality 
	$mp^2\le -m\ln(1-p^2)=\ln\left(1/({1-\hatp})\right)$
	to \eqref{eq:best}, by Corollary~\ref{Cor:Main} and Lemma~\ref{Lem:GIprawd},  we get 
	\begin{align*}
		d_{K/W}\left(\widetilde \Xe,   \mathcal N\right)
		&\lesssim	\left(\dfrac1{n^2\hatp(1-\hatp)}+\frac1n+\frac1m\ln\left(\frac1{1-\hatp}\right)\frac{(1-\hatp)}{\hatp}\right)^{1/2}.
	\end{align*}
	Since $\sup_{x\in(0, 1)}\ln(\tfrac1{1-x})\tfrac{1-x}x<\infty$. 
\end{proof}

We proceed with the proof of Lemma~\ref{Lem:GIprawd}. In the case $G_t\in \{G_{1,0,2},G_{1,1,2}\}$ it is implied by Lemma~\ref{lem:2norms}.
Let now $G_t\in \{G_{1,0,1},G_{2,1,2},G_{3,1,2}\}$.
In this case we will need to estimate the subgraph probabilities in $\Gnmp$, which is made using so--called clique covers. The notion of clique covers was introduced in \cite{1999KarSchSin}. We will use the results obtained in \cite{SubgraphProb} and follow the notation introduced therein.

For a graph $H$ we denote by  $\calP(H)$ the family of subsets of $V(H)$ that contain both ends of at least one edge of $H$. A family $\cover\subseteq \calP(H)$ is called a {\sl clique cover} of $H$ if, for every $e\in E(H)$, there exists a set $C\in\cover$ such that $e\subseteq C$.  Furthermore,  by $\covers=\covers(H)$ we denote the set of all clique covers of $H$. Additionally, we use the term 
$t$-set to refer to any set containing 
$t$ elements; this term will be mainly used in the context of vertex subsets.

For $H$  such that $V(H)\subseteq \V$ we say that $a\in \A$ {\sl builds} $C\subseteq V(H)$ if  all vertices from $C$ chose attribute $a$ and no vertex from $V(H)\setminus C$ chose $a$ in $\Gnmp$. The probability of such an event is given by
\begin{align}\label{eq:p_C}
	p_C:=p^{|C|}(1-p)^{|V(H)|-|C|}.
\end{align}
We say that $H$ is {\sl given by a clique cover} $\cover\subseteq \calP(H)$ in $\Gnmp$ if for every $C\in \cover$  there is an attribute in $\A$ that builds $C$ and no attribute from $\A$ builds any $C\in \calP(H)\setminus \cover$.

Let $\pi(H,  \cover)$ be the probability that $H$ is given by a clique cover $\cover$ in $\Gnmp$. For a subfamily of clique covers $\covers\subseteq \covers(H)$ we denote the probability that $H$ is given by a clique cover from   $\covers$ by
\begin{equation}\label{Eq:SubgraphProbabilityCCfamily}
	\pi(H,  \covers):=\sum_{\cover\in \covers} \pi(H,  \cover). 
\end{equation}
In particular,  the probability that $H$ is a subgraph of $\Gnmp$ is given by
\begin{align}\label{Eq:SubgraphProb}
	\pi(H)=\Pra{H\subseteq \Gnmp}=\pi(H,  \covers(H))
	=\sum_{\cover\in\covers(H)}\pi(H,  \cover). 
\end{align}
Next,  we define  families of clique covers that will be used to compare clique covers of different graphs. For a clique cover
$\cover_+\subseteq \calP (H)$ and a set
$\cover_-\subseteq \calP(H)\setminus \cover_+$
we denote by  $\covers(H,  \cover_+,  \cover_-)$ the family of  all the clique covers $\cover$ containing all sets from  $\cover_+$ and no set from $\cover_-$,   i.e.  such that $\cover_+\subseteq \cover\subseteq\calP(H)\setminus \cover_-$. The sets from $\cover_-$ will be called {\sl forbidden}. We will be using the shorthand notation
\begin{equation}\label{Eq:SubgraphProbabilityCCfamily2}
	\pi(H,  \cover_+,  \cover_-):=\pi(H,  \covers(H,  \cover_+,  \cover_-))=\sum_{\cover\in \covers(H,  \cover_+,  \cover_-)} \pi(H,  \cover). 
\end{equation}
As a special case, we have $\pi(H,  \cover)=\pi(H,  \cover,  \calP(H)\setminus\cover)$.

In Lemma \ref{Lem:CCprobability} we present a general result obtained in \cite{SubgraphProb}. A version of this lemma was proved in \cite{1999KarSchSin}, and was  generalised  in \cite{2017RybSta} and \cite{Dong2023}. 

\begin{lemma}\label{Lem:CCprobability}
	Let $H$ be a graph and   $\calP(H)$  be the family of subsets of $V(H)$ containing both ends of at least one edge of $H$. Fix a clique cover  $\cover_+$   and a set $\cover_-\subseteq \calP(H)\setminus \cover_+$.
	
	If $mp^3\le 1$ and $m>m_0$ for some $m_0\in\N$, the following estimates hold true 
	\begin{align}\label{eq:CCasymp}
		\pi(H, \cover_+,  \cover_-)
		&\asymp e^{-|\cover_-^{(2)}|mp^2}(1-e^{-mp^2})^{|\cover_+^{(2)}|}\prod_{C\in\cover_+\setminus \cover_+^{(2)}}mp^{|C|}, && 
	\end{align}
	where $\cover_+^{(2)}$ and $\cover_-^{(2)}$ are the families of 2--sets in $\cover_+$ and $\cover_-$,  respectively.
\end{lemma}


The next lemma was also proven in \cite{SubgraphProb}. It will allow us to compare the probabilities of graphs with similar, though not identical, structures.
Let $\cover$ and $\cover'$ be two families of sets. We write $\cover\Bumpeq\cover'$  if  there is a bijection $f:\cover\rightarrow \cover'$ such that $|f(C)|=|C|$ for any $C\in \cover$.

\begin{lemma}\label{Lem:CCequality}
	Let $H_1$ and $H_2$ be two graphs with the same number of vertices. Let,   for $i=1,  2$,   $\cover_+^{(i)}$ be a clique cover of $H_i$ and 
	$\cover_-^{(i)}\subseteq \calP(H_i)\setminus \cover_+^{(i)}$. If $\cover_+^{(1)}\Bumpeq\cover_+^{(2)}$  and $\cover_-^{(1)}\Bumpeq\cover_-^{(2)}$, then
	$$
	\pi(H_1,  \cover_+^{(1)},  \cover_-^{(1)})=\pi(H_2, \cover_+^{(2)},  \cover_-^{(2)}).
	$$
\end{lemma}

The next corollary plays a crucial role in the proof of  Lemma~\ref{Lem:HIJprawdopodobienstwa}.
\begin{corollary}\label{cor:piFuG}
	Let $F, F', G, G'$ be graphs with the same vertex set, and  satisfying $F\cong F'$, $|E(G)|=|E(G')|$ and $E(F\cap G)=E(F'\cap G')=\emptyset$. Then 
	$$\pi(F\cup G, \covers(F, G))=\pi(F'\cup G', \covers(F',G')),$$
	where $\covers(F, G)$ is the set of all clique covers $\cover$ of $F\cup G$ such that $E({G})\subseteq \cover\subseteq \covers(F\cup G)\setminus E(F)$.
\end{corollary}

\begin{proof}
	Denote by $\covers^{>2}(F)$ the family of all clique covers of $F$ that do not contain $2$--sets. In order to ensure the accuracy of our reasoning in the case $E(F)=\emptyset$, we set $\calP(F)=\{\emptyset\}$ and $ \covers^{>2}(F)=\{\emptyset\}$.  Then we have
	$$\covers(F,G)=\bigcup_{\cover\in \covers^{>2}(F)}\covers\left(F\cup G, \cover\cup E(G), \calP(F)\setminus \cover\right).$$
	Consider $\cover_1, \cover_2\in \covers^{>2}({F})$,  $\cover_1\neq \cover_2$, and set $\covers^{(i)}=\covers\left(F\cup G, \cover_i\cup E(G), \calP(F)\setminus \cover_i\right)$, $i=1,2$.
	Without loss of generality we may assume that there exists $C\in \cover_1\setminus\cover_2$. Then, for any $\cover\in \covers^{(1)}$ we have $C\in \cover$. On the other hand, $C$ is forbidden in any clique cover from $\covers^{(2)}$ and hence $C\not\in \cover$ for any $\cover\in \covers^{(2)}$. Consequently,   the families in the sum above are disjoint, which allows us to write
	\begin{align}\label{Eq42:CCsHIsuma}
		\pi(F\cup G, \covers(F, G))&=\sum_{\cover\in\covers (F, G)}\pi(F\cup G, \cover)\\ \nonumber
		&=\sum_{\cover\in \covers^{>2}(F)}\pi\big(F\cup G, \cover\cup E(G), \calP(F)\setminus \cover\big).
	\end{align}
	Let $\varphi:V(F)\to V(F')$ be an isomorphism of $F$ and $F'$.
	\eqref{Eq:HIIpromHIIbis}. Then there is a natural bijection $\psi:\covers^{>2}(F)\to \covers^{>2}(F')$ given by the formula $\psi(\cover)=\{\varphi(C): C\in \cover\}$, for $\cover\in \covers^{>2}(F)$. Clearly, for any $\cover\in \covers^{>2}(F)$, $\cover\Bumpeq\psi(\cover)$, and consequently
	\begin{align*}
		\cover\cup E(G)
		&\Bumpeq\psi(\cover)\cup E(G');\\
		\calP(F)\setminus\cover
		&\Bumpeq  \calP(F')\setminus\phi(\cover),	
	\end{align*}
	where we also used $\cover\cap E(G)= \psi(\cover)\cap E(G')=\emptyset$.
	Therefore, by Lemma~\ref{Lem:CCequality}, for any $\cover\in \covers^{>2}(F)$ we have
	\begin{align}
		&\pi(F\cup G,   \cover\cup E(G),  \calP(F)\setminus\cover)=
		\pi(F'\cup G',  \psi(\cover)\cup E(G'),  \calP(F')\setminus\psi(\cover)).
	\end{align}
	Eventually, using the above equality and the decomposition \eqref{Eq42:CCsHIsuma} we obtain
	\begin{align*}
		\pi(F\cup G, \covers(F, G))
		&=\sum_{\cover\in \covers^{>2}(F)}\pi\big(F\cup G, \cover\cup E(G), \calP(F)\setminus \cover\big)\\
		&=\sum_{\cover\in \covers^{>2}(F)}\pi(F'\cup G',  \psi(\cover)\cup E(G'),  \calP(F')\setminus\psi(\cover))\\
		&=\sum_{\cover\in \covers^{>2}(F')}\pi\big(F'\cup G', \cover\cup E(G'), \calP(F')\setminus \cover\big)\\
		&=\pi(F'\cup G', \covers(F', G')),
	\end{align*}
	as required.
\end{proof}




\begin{figure}[t]
	\begin{center}
		\includegraphics[width=350pt]{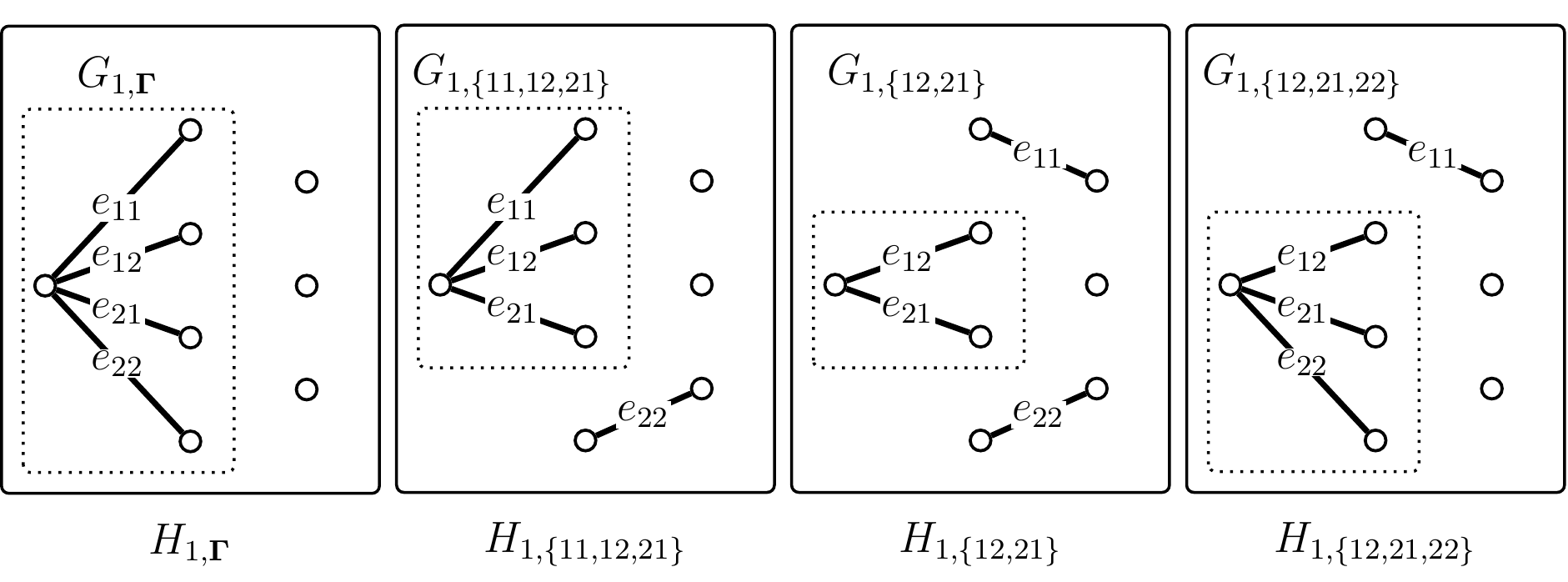}
	\end{center}
	\caption{\label{Fig:DefHindeksy} Examples of $H_{1,\II}$ for $G_1=G_{1,0,1}$, which is isomorphic to a star. $G_{1,\II}$ has been placed in the dotted rectangle.}
\end{figure}

Recall that $G_{t,\II}=G_t[\{e_{ab}:ab\in \II\}]$ is the subgraph of $G_t$ induced by the edges with indices in~$\II$. Let $H_{t,\II}$ be a graph on $8$ vertices ($V(H_{t,\II})\subseteq \V$) which consists of $G_{t,\II}$,   isolated edges $e_{ab}$,   $ab\in \indeksy\setminus \II$,   and possibly  some isolated vertices (for examples, see Figure~\ref{Fig:DefHindeksy}). Let us note that
\begin{equation}\label{Eq:GiHi}
	\pp^{|\indeksy|-|\II|}\pi(G_{t,\II})=\pi(H_{t,\II}),
\end{equation}
which follows from to fact  that  the isolated edges $e_{ab}$,   $ab\in \indeksy\setminus \II$,  appear independently of the edges of $G_{t,\II}$, and isolated vertices do not affect the subgraph probability. 
Therefore 
\begin{equation}\label{Eq:Calki:Ht}
	\sum_{\II\subseteq\indeksy}(-1)^{|\II|}\pp^{|\indeksy|-|\II|}\pi(G_{t,   \II})
	=\sum_{\II\subseteq\indeksy}(-1)^{|\II|}\pi(H_{t,   \II})
\end{equation}


For each $H_{t,\II}$,  $\II\subseteq \indeksy$, $t=1,2,3$, and  $\JJ\subseteq \indeksy$ we define \smallskip

\begin{center}
	$\covers_{\JJ}(H_{t,\II})$ -- the family of clique covers of $H_{t,\II}$ such that no edge $e_{ab}$,   $ab\in \JJ$,   is covered\\by a 2--set and all edges $e_{ab}$,   $ab\in\indeksy\setminus\JJ$ are covered by 2--sets.\ \ \ \ \ 
\end{center}
\smallskip

Note that $\covers_{\JJ}(H_{t,\II})$, $\JJ\subseteq \indeksy$,  determine a decomposition of the family of clique covers $\covers(H_{t,\II})$. Now we may expand the expression from  \eqref{Eq:Calki:Ht} as follows
\begin{multline}\label{Eq42:CalkaDoSumaHIJ}
	\sum_{\II\subseteq \indeksy}(-1)^{|\II|}\pi(H_{t,\II})
	=\sum_{\II\subseteq \indeksy}(-1)^{|\II|}\sum_{\JJ\subseteq\indeksy}\pi(H_{t,\II},  \covers_{\JJ}(H_{t,\II}))\\
	=\sum_{\emptyset\subseteq\JJ\subseteq\indeksy}\sum_{\II\subseteq \indeksy}(-1)^{|\II|}\pi(H_{t,\II},  \covers_{\JJ}(H_{t,\II})).
\end{multline}
Next, we will show that the inner sum vanishes for $\JJ\neq \indeksy$ and derive bounds for the inner sum for $\JJ=\indeksy$. Then, Lemma~\ref{Lem:GIprawd} will follow from the lemma below.
\begin{lemma}\label{Lem:HIJprawdopodobienstwa}For $ \emptyset\subseteq\JJ\subsetneq\indeksy$  and $t=1, 2, 3$ we have
	$$\sum_{\II\subseteq \indeksy}(-1)^{|\II|}\pi(H_{t,\II},  \covers_{\JJ}(H_{t,\II}))=0.$$ 
	Moreover, for $mp^3\le 1$ and $m\to\infty$ it holds that 
	\begin{multline*}
		n^{|V(G_t)|}\left|\sum_{\II\subseteq \indeksy}(-1)^{|\II|}\pi(H_{t,\II},  \covers_{\indeksy}(H_{t,\II}))\right|
		\\
		\lesssim
		(\Var[\Xe])^2\left(\frac{1}{n^2\pp(1-\pp)}+\frac1n+\left(\frac{1}{nmp}\wedge \frac1m \frac{mp^2(1-\hatp)}{\hatp}\right)\right).	
	\end{multline*}	
\end{lemma}
\begin{figure}[t]
	\begin{center}
		\includegraphics[width=160pt]{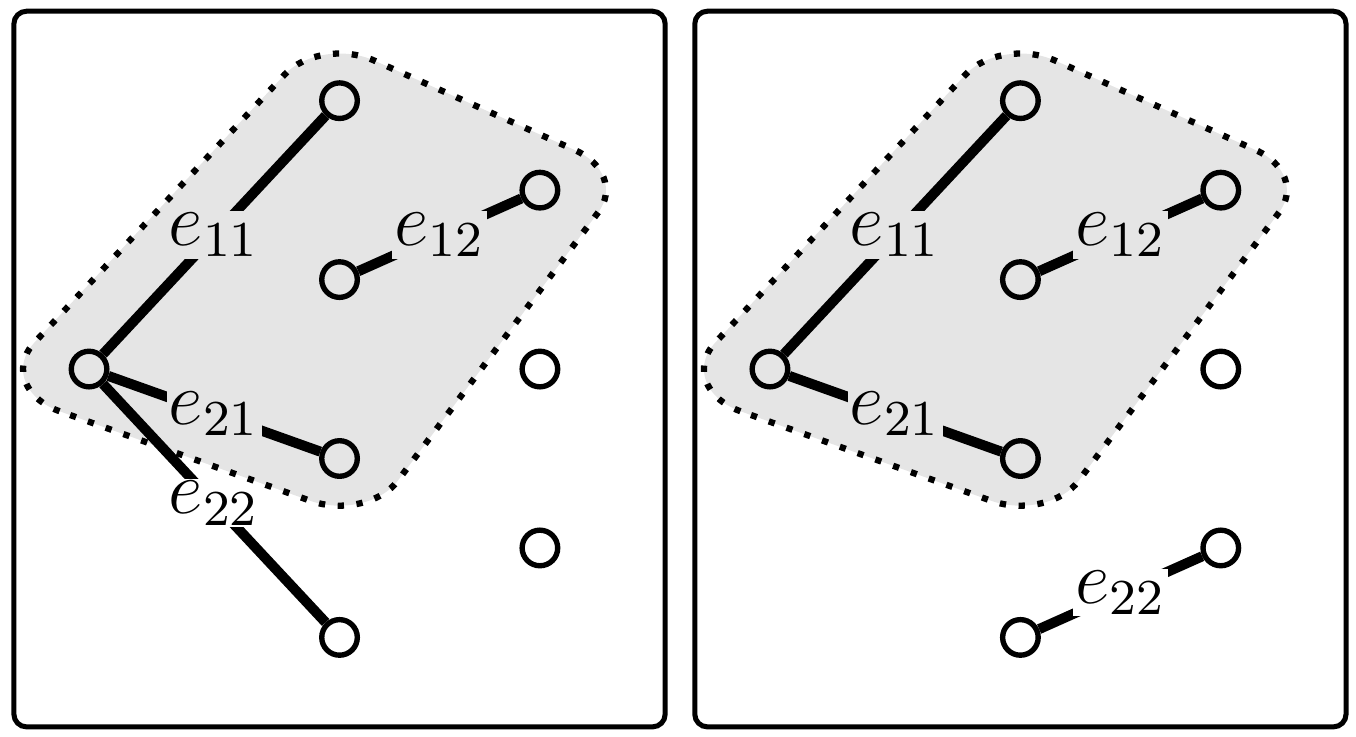}
	\end{center}
	\caption{\label{Fig:DefHIJbis} Let  $\JJ=\{11,  12, 21\}$ and $\JJ^*=\{11,21\}$. In the figure we show all $H_{1,\II}$ such that $\JJ\cap\II=\JJ^*$. By \eqref{Eq:HIIpromHIIbis} all $H_{1,\II, \JJ}$ (with the edge set indicated by the dotted line) are isomorphic. Each $H_{1,\II, \JJ}$ consist of $G_{\JJ^*}$ (built with edges $e_{11}$ and $e_{21}$), an isolated edge $e_{12}$, and three isolated vertices (those outside the dotted line).}
\end{figure}
\begin{proof}[Proof of Lemma~\ref{Lem:HIJprawdopodobienstwa}]
	Set $i=1, 2, 3$ and 
	let $\JJ$ such that $\JJ\subsetneq\indeksy$. Then 	
	\begin{align}\nonumber
		\sum_{\II\subseteq \indeksy}(-1)^{|\II|}\pi(H_{t,\II},  \covers_{\JJ}(H_{t,\II}))&
		=\sum_{\JJ^*\subseteq \JJ}\sum_{\substack{\II\subseteq\indeksy\\  \II\cap \JJ=\JJ^*}}(-1)^{|\II|}\pi(H_{t,\II},  \covers_{\JJ}(H_{t,\II}))\\
		&=\sum_{\JJ^*\subseteq \JJ}(-1)^{|\JJ^*|}\sum_{\substack{\II\subseteq\indeksy\\  \II\cap \JJ=\JJ^*}}(-1)^{|\II\setminus \JJ|}\pi(H_{t,\II},  \covers_{\JJ}(H_{t,\II})).\label{Eq42:SumaHIJdwa}
	\end{align}

	Thus, in order to prove  the first equation in Lemma~\ref{Lem:HIJprawdopodobienstwa},  it now suffices to show  that the inner sum in \eqref{Eq42:SumaHIJdwa} always vanishes, i.e.
	\begin{equation}\label{Eq42:SumaHIJgwiazdka}
		\forall_{\JJ\subsetneq\indeksy}\forall_{\JJ^*\subseteq \JJ}\ 
		\sum_{\substack{\II\subseteq\indeksy\\  \II\cap \JJ=\JJ^*}}(-1)^{|\II\setminus \JJ|}\pi(H_{t,\II},  \covers_{\JJ}(H_{t,\II}))=0.
	\end{equation}
	For any $\II\subseteq \indeksy$ we define 
	\begin{equation}\label{Eq:DefHIJ}
		H_{t,\II,  \JJ}=H_{t,\II}-\{e_{ab}: ab\notin \JJ\},
	\end{equation}
	which is the subgraph of $H_{t,\II}$ with the vertex set as $H_{t,\II}$ and only edges $e_{ab}$ with $ab\in \JJ$. 
	
	Next, we will prove that for any $\II', \II''\subseteq\indeksy$  such that $\II'\cap \JJ=\II''\cap \JJ = \JJ^*$ we have
	\begin{equation}\label{Eq:HIIpromHIIbis}
		H_{t,\II',  \JJ}\cong H_{t,\II'',  \JJ}.
	\end{equation}
	For better understanding, we present in Figure~\ref{Fig:DefHIJbis} two examples of $H_{t,\II,  \JJ}$ such that $\II\cap \JJ=\JJ^*$.

	By the definition of $H_{t,\II'}$, its subgraph induced on the edges $e_{ab}$,    $ab\in \II'$,   is isomorphic to $G_{t,\II'}$. Consequently, the subgraph of  $H_{t,\II',  \JJ}$, induced on the edges $e_{ab}$, $ab\in \II'\cap \JJ$,   is isomorphic to $G_{t,\II'\cap\JJ}=G_{t,\JJ^*}$, i.e. we have
	\begin{equation*}
		H_{t,\II',  \JJ}[\{e_{ab},  ab\in \JJ\cap \II'\}]\cong G_{t,\JJ\cap \II'}\cong G_{t,\JJ^*}
		\cong G_{t,\JJ\cap \II''} \cong H_{t,\II'',  \JJ}[\{e_{ab},  ab\in \JJ\cap \II''\}].
	\end{equation*}

	Moreover, the edges $e_{ab}$,   $ab\in \JJ\setminus\II'=\JJ\setminus\II''$,   are disjoint from all other edges in both $H_{t,\II',  \JJ}$ and $H_{t,\II'',  \JJ}$. Therefore, both $H_{t,\II',  \JJ}$ and $H_{t,\II'',  \JJ}$ have vertex sets of the same size, and are composed of $G_{t,\JJ^*}$ and $|\JJ\setminus\II'|=|\JJ\setminus\II''|$ disjoint edges. Thus \eqref{Eq:HIIpromHIIbis} follows.

	%
	%

	The isomorphism \eqref{Eq:HIIpromHIIbis} allows us to apply Corollary \ref{cor:piFuG} with $F=H_{t,\II',  \JJ}$, $F'=H_{t,\II'',  \JJ}$, $G=H_{t,\II',  \indeksy\setminus\JJ}$, $G'=H_{t,\II'',  \indeksy\setminus\JJ}$. 
	For this purpose, let us remind that $\covers_{\JJ}(H_{t,\II'})$ consists of all clique covers in which no edge of $F=H_{t,\II',  \JJ}$ is covered by a 2--set, and all edges of $G=H_{t,\II',  \indeksy\setminus\JJ}$ are covered by 2--sets. Therefore, using the notation from Corollary \ref{cor:piFuG} we have $\covers_{\JJ}(H_{t,\II'})=\covers(F,G)$ and, analogously, $\covers_{\JJ}(H_{t,\II''})=\covers(F',G')$. Thus, by Corollary \ref{cor:piFuG}, for any $\II'$ and $\II''$ such that $\II'\cap\II''=\JJ^*$,
	$$
	\pi(H_{t,\II'},  \covers_{\JJ}(H_{t,\II'}))
	=
	\pi(F\cup G,\covers(F,G))
	=
	\pi(F'\cup G',\covers(F',G'))
	=
	\pi(H_{t,\II''},  \covers_{\JJ}(H_{t,\II''}))
	$$
	
	Consequently,  all terms $\pi(H_{t,\II},  \covers_{\JJ}(H_{t,\II}))$ in the sum in \eqref{Eq42:SumaHIJgwiazdka}  have the same value, which gives us
	\begin{align*}
		\sum_{\substack{\II\subseteq\indeksy\\  \II\cap \JJ=\JJ^*}}(-1)^{|\II\setminus \JJ|}\pi(H_{t,\II},  \covers_{\JJ}(H_{t,\II}))
		&=\pi(H_{t,\JJ^*},  \covers_{\JJ}(H_{t,\JJ^*}))\sum_{\substack{\II\subseteq\indeksy\\  \II\cap \JJ=\JJ^*}}(-1)^{|\II\setminus \JJ|}\\
		&=\pi(H_{t,\JJ^*},  \covers_{\JJ}(H_{t,\JJ^*}))\sum_{k=0}^{|\indeksy\setminus \JJ|}(-1)^{k}\binom{|\indeksy\setminus \JJ|}{k}=0.
	\end{align*}
	This proves \eqref{Eq42:SumaHIJgwiazdka}, thereby completing the proof of the first equation in  Lemma~\ref{Lem:HIJprawdopodobienstwa}.

	What has left to prove is the latter assertion in Lemma~\ref{Lem:HIJprawdopodobienstwa}.
	Recall that, for any $\II\subseteq\indeksy$, by the definition of $\covers_{\indeksy}(H_{t,\II})$, it contains all clique covers of $H_{t,\II}$ with no 2--sets. 

	First, consider clique covers from $\covers_{\indeksy}(H_{1,\II})$, $\II\subseteq \indeksy$.  
	Each such a clique cover must contain at least one $k$--sets,   $k\ge 5$,   or at least $2$ sets of size at least~$3$. Moreover, none of the considered clique covers has any 2--set. Therefore, by \eqref{eq:CCasymp} and the assumption $mp^3\leq 1$, we obtain
	\begin{equation*}
		\forall_{\II}\forall_{\cover\in \covers_{\indeksy}(H_{1,\II})}\  \pi(H_{1,\II},  \cover)
		\lesssim (mp^5+(mp^3)^2)e^{-4mp^2}.
	\end{equation*}
	As a result, by \eqref{Eq:SubgraphProbabilityCCfamily} we have
	\begin{equation}\label{Eq:Calka1ograniczeniemp}
		\left|\sum_{\II\subseteq \indeksy}(-1)^{|\II|}\pi(H_{1,\II},  \covers_{\indeksy}(H_{1,\II}))\right|
		\le \sum_{\II\subseteq \indeksy}\pi(H_{t,\II},  \covers_{\indeksy}(H_{1,\II})) 
		\lesssim (mp^5+(mp^3)^2)e^{-4mp^2}.
	\end{equation}
	For $H_{3,\II}$, $\II\subseteq \indeksy$, similarly as for $H_{1,\II}$, all clique covers with no 2--sets either  contain one $k$--sets,   $k\ge 5$,   or at least $2$ sets of size at least $3$.
	Therefore,  by \eqref{eq:CCasymp} and \eqref{Eq:SubgraphProbabilityCCfamily}  we get
	\begin{equation}\label{Eq:Calka2ograniczeniemp}
		\left|\sum_{\II\subseteq \indeksy}(-1)^{|\II|}\pi(H_{3,\II},  \covers_{\indeksy}(H_{3,\II}))\right|\lesssim (mp^5+(mp^3)^2)e^{-4mp^2}.
	\end{equation}
	In the case of $H_{2,\II}$, $\II\subseteq \indeksy$, all the clique covers that do not include 2--sets use either at least one $k$--set,   $k\ge 4$,   or at least two sets with at least 3 vertices.
	Therefore
	
	\begin{align*}
		\forall_{\II}\forall_{\cover\in \covers_{\indeksy}(H_{2,\II})}
		\pi(H_{2,\II},  \covers_{\indeksy}(H_{2,\II}))&\lesssim (mp^4+(mp^3)^2)e^{-4mp^2}. 
	\end{align*}
	This leads to
	\begin{equation}\label{Eq:Calka3ograniczeniemp}
		\left|\sum_{\II\subseteq \indeksy}(-1)^{|\II|}\pi(H_{2,\II},  \covers_{\indeksy}(H_{2,\II}))\right|\lesssim (mp^4+(mp^3)^2)e^{-4mp^2}.
	\end{equation}


	We now derive bounds for the expressions on the right-hand sides of equations \eqref{Eq:Calka1ograniczeniemp}, \eqref{Eq:Calka2ograniczeniemp}, and \eqref{Eq:Calka3ograniczeniemp}, after multiplying them by the factor $n^{|V(G_t)|}$. This will complete the proof of Lemma~\ref{Lem:HIJprawdopodobienstwa}.

	We will use the facts that $mp^2\ge \pp\ge \pp(1-\pp)$ and $1-\pp\asymp e^{-mp^2}$, together with \eqref{eq:var-mp3small}.
	\begin{align*}
		n^4mp^4e^{-4mp^2}
		=\frac{\big(n^3mp^3e^{-2mp^2}\big)^2}{n^2mp^2}
		\lesssim\frac1{n^2\hatp (1-\hatp)} (\Var[\Xe])^2.
	\end{align*}
	Similarly,  
	\begin{align*}
		n^5mp^5e^{-4mp^2}
		&=\frac1m{(n^3mp^3e^{-2mp^2})(n^2mp^2 e^{-2mp^2})}\\
		&\asymp\frac{(n^3mp^3e^{-2mp^2})\Var [\Xe]}m\cdot {\frac{n^2mp^2 e^{-2mp^2}}{n^3mp^3(1-\hatp)^2+n^2\hatp(1-\hatp)}}\\
		&\lesssim (\Var [\Xe])^2\frac1m\cdot \frac{mp^2(1-\hatp)}{nmp^3(1-\hatp)+\hatp}\\
		&\asymp (\Var [\Xe])^2\frac1m\left(\frac{1}{np}\wedge \frac{mp^2(1-\hatp)}{\hatp}\right)
	\end{align*}	
	and
	\begin{align*}
		n^5(mp^3)^2 e^{-4mp^2}
		&=\frac1n{\big(n^3mp^3e^{-2mp^2}\big)^2}
		\lesssim (\Var [\Xe])^2+\frac1n.
	\end{align*}	
\end{proof}

\section{The case $mp^3\ge 1$}\label{sec:MainProof}

Now we turn our attention to large values of $mp^3$. We approach the problem in this case by considering the number of the non-existing edges. In particular,  this helps us to determine the precise values of some norms that  appear in the distance bounds.

Additionally,  we define the function $\varrho=1-g$. The variable 
\begin{align}\label{eq:V}
	\Ze=\sum_{1\le k<l\le n}\varrho(A^{(k)},  A^{(l)})
\end{align}
counts the number of edges that do not appear in $\Gnmp$ or,  in other words,  the number of edges in the complement $\cGnmp$ of the graph $\Gnmp$. In particular,  we have $\Ze={n\choose 2}-\Xe$,  and therefore $\widetilde \Xe=-\widetilde \Ze$. Due to the symmetry of the standard normal distribution,  we have
\begin{align}\label{eq:dN=dV}
	d_{K/W}(\widetilde \Xe,  \mathcal N) =d_{K/W}(\widetilde \Ze,  \mathcal N), 
\end{align}
so we may equivalently consider $\Ze$ instead of $\Xe$ in the context of normal approximation,  which   is beneficial for some ranges of parameters.

Similarly as in the case of the function $g$,  it holds that
\begin{align*}
	\varrho_2(x,  y)&=\varrho(x,  y), \\
	\overline \varrho_2(x,  y)&=\varrho(x,  y)-
	(1-\hatp)=-\overline g_2(x, y), \\
	\varrho_1(x)&=\int_{\{0,  1\}^m}\varrho(x,  y)d\mu_{m,  p}(y)=1-g_1(x),  \\
	\overline \varrho_1(x)&=\varrho_1(x)-(1-\hatp)=-\overline g_1(x).
\end{align*}
Consequently,  one can replace any $g$ in \eqref{Eq:dkW} with $\varrho$ without changing the value of the bound. Analytically,  considering the function $\varrho$ does therefore not change anything,  however,  it allows us to look at the appearing expressions differently from probabilistic point of view. Let us note that such an approach  already turned out to be  convenient when evaluating the probability in \eqref{eq:ee}.

Let us recall that  $\cGnmp$ denotes  the complement of $\Gnmp$,  and by the definition we have $\varrho=1-g$.  Furthermore,  let  $K_{1,  4}$ be the star with $4$ leaves on some fixed vertices from $\mathcal V$.

\begin{lemma}\label{lem:normsrho} For $n\geq 5$ we have
	\begin{align}\label{eq:K14}
		\left\| \varrho_{1}\ast_{1}^0 \varrho_{1}\right\|_{2}^2&=\P\left(K_{1, 4}\subseteq \cGnmp\right)=\big(1-p+p(1-p)^4\big)^m, 
		\\[7pt]
		\left\| \varrho_{2}\ast_{1}^1 \varrho_{2}\right\|_{2}^2
		&\label{eq:C4}=\big((1-p)^4+4p(1-p)^3+2p^2(1-p)^2\big)^m, 
		\\[7pt]
		\left\| \varrho_{2}\ast_{1}^1 \varrho_{1}\right\|_{2}^2
		\label{eq:P5}
		&=\big((1-p)^5+5p(1-p)^4+6p^2(1-p)^3+p^3(1-p)^2\big)^m, \\[7pt]
		\label{eq:P3}
		\|\varrho_1\|_2^2&
		=\left(1-2p^2+p^3\right)^m.
	\end{align}
\end{lemma}
\begin{proof}
	We will interpret the norms analogously as it was presented in Section \ref{sec:IntegralsToGraphs}  for the \mbox{kernel $g$.}  First,  note that we have the isomorphisms $G_1\cong K_{1, 4}$,  $G_2\cong C_4$,  and $G_3\cong P_5$ (c.f. Figure \ref{Fig:GtExamples}),  where $C_4$ is a cycle on $4$ vertices and $P_5$ is a path on $5$ vertices.  Therefore,  conducting the same calculations as in the proof of Lemma~\ref{Lem:CalkiDoPodgrafy},  we have
	\begin{align*}
		\left\| \varrho_{1}\ast_{1}^0 \varrho_{1}\right\|_{2}^2=\P\left(K_{1, 4}\subseteq \cGnmp\right)=\big(1-p+p(1-p)^4\big)^m.
	\end{align*}
	The latter equality follows from the fact that the event $K_{1, 4}\subseteq \cGnmp$ occurs if and only if for any attribute $a\in \mathcal A$  and for any two  vertices  adjacent  in $K_{1, 4}$ it does not happen that both of them chose $a$. It is equivalent to the event that either the centre did not choose $a$ or  it did and none of the leaves did. 
	
	Similarly we get
	\begin{align*}
		\left\| \varrho_{2}\ast_{1}^1 \varrho_{2}\right\|_{2}^2
		&=\P\left(C_4\subseteq \cGnmp\right)
		=\big((1-p)^4+4p(1-p)^3+2p^2(1-p)^2\big)^m, 
		\intertext{
			where the last expression comes from consideration of cases regarding number of vertices associated to an attribute. In the same manner
			we obtain}
		\left\| \varrho_{2}\ast_{1}^1 \varrho_{1}\right\|_{2}^2
		&=\P\left(P_5\subseteq \cGnmp\right)\\
		&=\big((1-p)^5+5p(1-p)^4+6p^2(1-p)^3+p^3(1-p)^2\big)^m.
	\end{align*}
	Eventually
	\begin{align*}
		\|\varrho_1\|_2^2&=\int_{(\{0, 1\}^m)^3}\varrho(x_1,  x_2)\varrho(x_1, x_3)d\mu_{m, p}^{\otimes 3}(x)\\&=\P\left(P_3\subseteq \cGnmp\right)=\left(1-2p^2+p^3\right)^m, 
	\end{align*}
	which ends the proof.
\end{proof}

We are now prepared to  bound the distance $d_{K/W}(\widetilde \Xe,   \mathcal{N})$  for large $mp^3$.
\begin{proposition}\label{prop:K14} For $n\geq5$ and $mp^3\geq1$ we have
	\begin{align}\label{eq:K14++}
		d_{K/W}(\widetilde \Xe ,   \mathcal{N})
		&\lesssim \left(\frac1{n^2\hatp(1-\hatp)}+\frac1n+\frac{n^{5}\, {\P\left(K_{1, 4}\subseteq \cGnmp\right)}}{(\Var [\Xe])^2}\right)^{1/2}.
	\end{align}
\end{proposition}

\begin{proof} 
	The first step is to derive a version of \eqref{Eq:dkW},  where the last three norms are expressed by means of the kernel $\varrho$ instead $\overline g$. One can do this by thorough estimation,  however,  it is more convenient to go back to Theorem \ref{thm:generaldb}. Namely,  by \eqref{eq:dN=dV} we have
	\begin{align*}
		d_{K/W}(\widetilde \Xe ,   \mathcal{N})
		&\le \frac{C'}{\Var [\Xe]} 
		\Bigg(
		{n^2\left\|\db \varrho_{2}\ast_{2}^0 \db \varrho_{2}\right\|_{2}^2}
		+{n^{3}\left\|\db \varrho_{2}\ast_{2}^1 \db \varrho_{2}\right\|_{2}^2}+{n^{5}\left\| \db \varrho_{1}\ast_{1}^0 \db \varrho_{1}\right\|_{2}^2}\\
		&\phantom{\le \frac{C'}{\Var [\Xe]} 
			\Bigg(
			{n^2\left\|\db \varrho_{2}\ast_{2}^0 \db \varrho_{2}\right\|_{2}^2}}
		+
		{n^4\left\| \db \varrho_{2}\ast_{1}^1 \db \varrho_{2}\right\|_{2}^2}
		+
		{n^{5}\left\| \db \varrho_{2}\ast_{1}^1 \db \varrho_{1}\right\|_{2}^2}
		\Bigg)^{1/2}.
	\end{align*}
	Next,  we apply Lemma \ref{lem:db<} directly to  the last three norms,  and identifying $\db \varrho_2=\db{-\overline g_2}$ in the other ones.  This gives us
	\begin{align}\nonumber
		&d_{K/W}(\widetilde \Xe ,   \mathcal{N})\lesssim\\
		&\lesssim \frac{1}{\Var [\Xe]} 
		\Bigg(
		{n^2\left\|\overline g_{2}\ast_{2}^0 \overline g_{2}\right\|_{2}^2}
		+{n^{3}\left\|\overline  g_{2}\ast_{2}^1 \overline  g_{2}\right\|_{2}^2}+
		{n^{5}\left\| \varrho_{1}\ast_{1}^0 \varrho_{1}\right\|_{2}^2}
		+
		{n^4\left\| \varrho_{2}\ast_{1}^1 \varrho_{2}\right\|_{2}^2}\\\nonumber
		&\phantom{\le \frac{C'}{(\Var [\Xe])^2} \Bigg(}\ \ +
		{n^{5}\left\| \varrho_{2}\ast_{1}^1 \varrho_{1}\right\|_{2}^2}+n^5 \|\varrho_{1}\|^2_2\bigg(\int_{(\{0,  1\}^m)^r}\varrho(x)d\mu_{m,  p}^{\otimes 2}(x)\bigg)^2\, 
		\Bigg)^{1/2}\\
		\nonumber
		&\lesssim \Bigg[\frac1{n^2\hatp(1-\hatp)}+\frac1n+\frac{1}{(\Var [\Xe])^2} \times\\\nonumber
		&\ \ \ \ \ \ \times\Bigg({n^{5}\left\| \varrho_{1}\ast_{1}^0 \varrho_{1}\right\|_{2}^2}
		+
		{n^4\left\| \varrho_{2}\ast_{1}^1 \varrho_{2}\right\|_{2}^2}
		+
		{n^{5}\left\| \varrho_{2}\ast_{1}^1 \varrho_{1}\right\|_{2}^2}+n^5(1-\hatp)^2\|\varrho_1\|_2^2
		\Bigg)\Bigg]^{1/2}, 
	\end{align}
	where we used Lemma~\ref{lem:2norms} to obtain the latter inequality. Since $1-\hatp= (1-p^2)^m\le e^{-mp^2}\le e^{-1}$ for $mp^3\geq1$ and due to \eqref{eq:K14},   
	it now suffices  to show the following inequalities  
	\begin{align*}
		{\left\| \varrho_{2}\ast_{1}^1 \varrho_{1}\right\|_{2}^2}&\le 
		\left\| \varrho_{1}\ast_{1}^0 \varrho_{1}\right\|_{2}^2, \\
		\frac{n^5(1-\hatp)^2\|\varrho_1\|_2^2}{(\Var [\Xe])^2}&\lesssim \frac1{n}, \\
		\frac{n^4\left\| \varrho_{2}\ast_{1}^1 \varrho_{2}\right\|_{2}^2}{(\Var [\Xe])^2}&\lesssim \frac1{n^2(1-\hatp)}.
	\end{align*}
	The inequality $ {\left\| \varrho_{2}\ast_{1}^1 \varrho_{1}\right\|_{2}^2}\le 
	\left\| \varrho_{1}\ast_{1}^0 \varrho_{1}\right\|_{2}^2$ comes directly from \eqref{eq:K14} and \eqref{eq:P5},  since
	\begin{align*}
		1-p+p(1-p)^4&=(1-p)[(1-p)+p]^4+p(1-p)^4\\
		&=(1-p)^5+5p(1-p)^4+6p^2(1-p)^3+4p^3(1-p)^2+p^4(1-p), 
	\end{align*}
	which might be also obtained probabilistically by considering cases where a given attribute is chosen by  different number of vertices in $K_{1, 4}$ without being chosen by both of the vertices of any edge. Furthermore,  by \eqref{eq:C4},  \eqref{eq:P3} and \eqref{eq:varest}  we get
	\begin{align*}
		\frac{n^5(1-\hatp)^2\|\varrho_1\|_2^2}{(\Var [\Xe])^2}&\lesssim \frac{n^5(1-p^2)^{2m}(1-2p^2+p^3)^m}{\big(n^3(1-2p^2+p^3)^m\big)^2}=\frac1{n}\left(\frac{1-2p^2+p^4}{1-2p^2+p^3}\right)^m\le \frac1{n}, 
	\end{align*}
	as well as
	\begin{align*}
		\frac{n^4\left\| \varrho_{2}\ast_{1}^1 \varrho_{2}\right\|_{2}^2}{(\Var [\Xe])^2}&\lesssim \frac{n^4\big((1-p)^4+4p(1-p)^3+2p^2(1-p)^2)\big)^m}{\big(n^3(1-2p^2+p^3)^m\big)^2}\\
		&=\frac1{n^2(1-\hatp)} \left(1-\Big(\frac{p}{1+p-p^2}\Big)^2\, \right)^m\le \frac1{n^2(1-\hatp)}, 
	\end{align*}
	which ends the proof.
\end{proof}
The last term in \eqref{eq:K14++} diverges at $p = 1$, which makes the bound ineffective for values of $p$ close to~$1$. A more detailed discussion of this issue has already been provided in Section~\ref{sec:mainresults}, following formula~\eqref{aux10}. On the other hand,  for $p$ close to zero one can effectively bound the right-hand side of \eqref{eq:K14++}  in various ways. We propose the following one,  which correspond to the necessary condition for the asymptotic normality of $\widetilde \Xe$.
\begin{corollary}\label{cor:mp^3>1} For $mp^3\geq 1$ and $p\le 0.1$ we have
	\begin{align*}
		d_{K/W}(\widetilde \Xe ,   \mathcal{N})&\lesssim\frac{1}{\big({n^2(1-\hatp)}\big)^{1/4}}.
	\end{align*}
\end{corollary}
\begin{proof}  From Proposition \ref{prop:K14},  Lemma \ref{lem:normsrho} and estimate \eqref{eq:varest} we deduce for $mp^3\geq 1$
	\begin{align}\label{aux8}
		d_{K/W}(\widetilde \Xe ,   \mathcal{N})
		\lesssim \[\frac1{n^2\hatp(1-\hatp)}+\frac1n+\frac1n\left(\frac{1-p+p(1-p)^4}{(1-2p^2+p^3)^2}\right)^m\]^{1/2}.
	\end{align}
	Thus,  by virtue of \eqref{eq:d<2} and the fact that $\hatp$ is bounded away from zero for $mp^3\geq1$,  it suffices to show that for $p\le0.1$
	$$Q(m,  p):=\left(\frac{1-p+p(1-p)^4}{(1-2p^2+p^3)^2}\right)^m\le\frac{1}{{(1-\hatp)}^{1/2}}. $$ 
	First,  we rewrite 
	\begin{align*}
		Q(m,  p)=\left(1+4p^3-p^4\, \frac{8 - 13p - 4 p^2 + 12 p^3 - 4 p^4}{(1 - p) (1 + p - p^2)^2}\right)^m.
	\end{align*}
	One can see that for $p<1/2$ the last ratio is positive and hence
	$Q(m,  p)\le \left(1+4p^3\right)^m$, $p\in(0, \tfrac12)$.	
	Furthermore,  it holds that
	\begin{align*}
		(1+4p^3)\sqrt{1-p^2}=\sqrt{1 - p^2\left(1 - 8 p + 8 p^3 - 16 p^4 + 16 p^6\right)},  
	\end{align*}
	which is clearly less than one for $p\le 0.1$ (numerical verification shows that it stays true for $p\le 0.126521\ldots$). This gives us
	\begin{align*}
		Q(m,  p)&\le \frac1{\sqrt{1-p^2}^{\, m}}=\frac1{\sqrt{1-\hatp}}, \ \ \ \ \ p\le0.1, 
	\end{align*}
	as required.
\end{proof}

\section{Asymptotic normality conditions}\label{sec:Sufficient}

Eventually,   we turn our attention to the  necessary and sufficient conditions for asymptotic normality of $\widetilde \Xe=(\Xe-\E[\Xe])/\sqrt{\Var[\Xe]}$.
Let us recall that   whenever the Kolmogorov or Wasserstein  distance  between random variables $X_n$ and a random variable $X$ tends to zero  then $X_n$ converges to $X$ in distribution. Thus,   in view of  Lemma~\ref{lem:mp^3<1},  Corollary \ref{cor:mp^3>1}   and the assumptions $n,  m\rightarrow\infty$,   the condition  $n^2\hatp(1-\hatp)\rightarrow\infty$  ensures $\widetilde \Xe\stackrel{\mathcal D}{\longrightarrow}\mathcal N$ whenever $p\leq0.1$. It turns out that the opposite implication holds true without any additional conditions on $p$.

\begin{lemma}\label{lem:necessary}
	If $\widetilde \Xe\stackrel{\mathcal D}{\longrightarrow}\mathcal N$,  then $n^2\hatp(1-\hatp)\rightarrow\infty$.
\end{lemma}
\begin{proof}
	We adapt the argument from \cite[proof of Theorem 3.5]{MNS}. First of all,   $\Var[\Xe]$ tends to infinity or else the limit would not be a continuous random variable.  
	Since
	$$\widetilde \Xe=\frac{ \Xe-\E[\Xe]}{\sqrt{\Var[\Xe]}}\geq -\frac{ \E[\Xe]}{\sqrt{\Var[\Xe]}},  $$
	and  the normal distribution is supported on the whole real line,   the last expression has to tend to minus infinity. Thus,   due to $\Var[\Xe]\rightarrow\infty$,   we get $\E[\Xe]={n\choose 2}\hatp\asymp n^2\hatp\rightarrow \infty$.\\
	On the other hand,   denoting $\Ze={n\choose 2}-\Xe$ we obtain  $\widetilde \Ze=-\widetilde \Xe\rightarrow \mathcal N$ and consequently
	$$\widetilde \Ze=\frac{ \Ze-\E[\Ze]}{\sqrt{\Var[\Ze]}}\geq -\frac{ \E[\Ze]}{\sqrt{\Var[\Xe]}}.$$
	Repeating the previous argument,  we conclude $\E[\Ze]={n\choose 2}(1-\hatp)\asymp n^2(1-\hatp)\rightarrow \infty$. Eventually,  the observation $\hatp\wedge (1-\hatp)\asymp \hatp (1-\hatp)$ completes the proof.
\end{proof}
\begin{remark}
	Theorem \ref{thm:Main} follows now from   Lemma  \ref{lem:mp^3<1}, Corollary \ref{cor:mp^3>1}, and Lemma~\ref{lem:necessary} .
\end{remark}
At the end of this section we present a threshold-like phenomenon for asymptotic normality,  that gives us some understanding what happens  when $p$ is bounded away from zero and from one,  without restricting it to be less than $0.1$.
\begin{proposition}\label{prop:mlnn}
	Let $p=p(n)\in\left(\varepsilon,  1-\varepsilon\right)$ for some fixed $\varepsilon\in(0, \tfrac12)$. 
	\begin{enumerate}
		\item If $\frac{m}{\ln n}\rightarrow0$,  then  $\widetilde \Xe\stackrel{\mathcal D}{\rightarrow}\mathcal N$.
		\item If $\frac{m}{\ln n}\rightarrow\infty$,  then  $\widetilde \Xe\stackrel{\mathcal D}{\not\rightarrow}\mathcal N$.
	\end{enumerate}
\end{proposition}
\begin{proof}Ad.{\it 1.} The assumption $p>\varepsilon$ implies $mp^3\rightarrow\infty$,  so it suffices to  
	show that both of the terms in \eqref{aux8} tend to zero. Indeed,  we have
	\begin{multline}\label{aux9}
		\frac1{n^2\hatp(1-\hatp)}= \frac1{1-(1-p^2)^m}n^{-\frac{m}{\ln n}\ln\left(1-p^2\right)-2}\\
		\le \frac1{1-(1-\varepsilon^2)^m}n^{-\frac{m}{\ln n}\ln\left(1-(1-\varepsilon)^2\right)-2}\rightarrow 0.
	\end{multline}
	Furthermore, 
	\begin{multline*}
		\frac1n\left(\frac{1-p+p(1-p)^4}{(1-2p^2+p^3)^2}\right)^m= \frac1n\left(\frac{1+p(1-p)^3}{(1-p)\big(1+p(1-p)\big)^2}\right)^m\\ \le \frac1n\left(\frac2{\varepsilon}\right)^m=n^{\frac{m}{\ln n}\ln\left(2/\varepsilon\right)-1}\rightarrow0, 
	\end{multline*}
	as required. 
	
	Ad.{\it 2.} If $\frac{m}{\ln n}\rightarrow\infty$,  then the expression in \eqref{aux9} tends to $\infty$,  and thus\linebreak ${n^2\hatp(1-\hatp)}\rightarrow 0$. Consequently,  
	the necessary condition from Lemma \ref{lem:necessary} is not satisfied.
\end{proof}

\section*{Acknowledgements}	

Grzegorz Serafin's work was supported by the National Science Centre (NCN), Poland,  under  research grant no. 2021/43/D/ST1/03244. The funder had no role in study design, data collection and analysis, decision to publish, or preparation of the manuscript.

\end{document}